%% file: arXiv-Jan24-20.tex
\RequirePackage{ifpdf} 
\ifpdf
\documentclass[11pt,pdftex]{article}
\else
\documentclass[11pt,dvips]{article}
\fi

\ifpdf
\usepackage{lmodern}
\fi
\usepackage[bookmarks, 
colorlinks=false, backref,plainpages=false]{hyperref}
\usepackage[T1]{fontenc}
\usepackage[latin1]{inputenc}
\usepackage[english]{babel}
\usepackage{graphicx}
\usepackage{url}
\usepackage{relsize}
\usepackage{graphics}
\usepackage{epsfig}
\usepackage{amsmath,amssymb,amsthm,mathrsfs}
\usepackage{graphics}
\usepackage{epsfig}
\usepackage{amsmath,amssymb,amsthm}
\usepackage{algorithm, algorithmic}

\usepackage{color}
\usepackage{bigints}

\renewcommand{\theequation}{\thesection.Arabic{equation}}

\newtheorem{thm}{Theorem}[section]
\newtheorem{cor}[thm]{Corollary}
\newtheorem{lem}[thm]{Lemma}

\newtheorem{rem}[thm]{Remark}

\begin{document}
\input defs.tex

\def\calS{{\cal S}}
\def\calT{{\cal T}}
\def\cB{{\cal B}}
\def\cH{{\cal H}}
\def\ba{{\mathbf{a}}}
\def\cN{{\cal N}}
%

\newcommand{\argmin}{\mbox{argmin}}

\def\bE{{\bf E}}
\def\bS{{\bf S}}
\def\br{{\bf r}}
\def\bW{{\bf W}}
\def\bphi{{\bf \phi}}

\newcommand{\lJump}{[\![}
\newcommand{\rJump}{]\!]}
\newcommand{\jump}[1]{[\![ #1]\!]}

\newcommand{\dd}{\underline{{\mathbf d}}}
\newcommand{\C}{\rm I\kern-.5emC}
\newcommand{\R}{\rm I\kern-.19emR}
\newcommand{\W}{{\mathbf W}}
\def\3bar{{|\hspace{-.02in}|\hspace{-.02in}|}}

\title {Generalized Prager-Synge Inequality and \\ 
Equilibrated Error Estimators for Discontinuous Elements}
\author{
Zhiqiang Cai\thanks{
Department of Mathematics, Purdue University, 150 N. University
Street, West Lafayette, IN 47907-2067, \{caiz\}@purdue.edu.
This work was supported in part by the National Science Foundation
under grant DMS-1522707.}
\and Cuiyu He\thanks{Department of Mathematics, 
University of Georgia,  1023 D.W. Brooks Dr, Athens, GA, 30605, USA,
cuyu.he@uga.edu.}
 \and Shun Zhang\thanks{Department of Mathematics, 
City University of Hong Kong, Hong Kong SAR, China,
shun.zhang@cityu.edu.hk.
This work was supported in part
by Hong Kong Research Grants Council under 
the GRF Grant Project No. 11305319, CityU 9042864.}
}
 \date{\today}
 \maketitle

\begin{abstract}
The well-known Prager-Synge identity is valid in $H^1(\Omega)$ and 
serves as a foundation for developing equilibrated a posteriori error estimators for continuous elements. 
In this paper, we introduce a new inequality, that may be regarded as a generalization of the Prager-Synge identity,
to be valid for {\it piecewise} $H^1(\Omega)$ functions for diffusion problems. 
The inequality is proved to be identity in two dimensions. 

For nonconforming finite element approximation of arbitrary odd order, 
we propose a fully explicit approach that recovers an equilibrated flux in $H(\divvr; \O)$
through a local element-wise scheme and that recovers a gradient in $H(\curll;\Omega)$
through a simple averaging technique over edges. 
The resulting error estimator is then proved to be globally reliable and locally efficient.
Moreover, the reliability and efficiency constants are independent of the jump of the diffusion coefficient  
regardless of its distribution.
\end{abstract}

\section{Introduction}\label{intro}
\setcounter{equation}{0}
Equilibrated a posteriori error estimators have attracted much interest recently due to the guaranteed reliability
bound with the reliability constant being one. This property implies that they are perfect for discretization
error control on both coarse and fine meshes. Error control on coarse meshes is important but difficult for computationally challenging problems.

For the conforming finite element approximation, a mathematical foundation of equilibrated estimators is
the Prager-Synge identity \cite{PrSy:47} that is valid in $H^1(\Omega)$ (see Section 3).
Based on this identity, 
various equilibrated estimators have been studied recently by many researchers 
(see, e.g.,  \cite{LaLe:83, DeSw:87, OdDeRaWe:89,
DeMe:98, DeMe:99, AiOd:00, AiOd:93, Vej:06, Br:07, BrSc:08, BrPiSc:09, Ve:09, CaZh:11, CaCaZh:20}). 
The key ingredient of the equilibrated estimators for the continuous elements is 
local recovery of an equilibrated (locally conservative) flux in the $H(\divvr;\O)$ space through the numerical flux.
By using a partition of unity, Ladev\`eze and Leguillon \cite{LaLe:83} initiated a local procedure 
to reduce the construction of an equilibrated flux to vertex patch based local calculations. 
For the continuous linear finite element approximation to the Poisson equation in two dimensions,
an equilibrated flux in the lowest order Raviart-Thomas space was explicitly constructed in \cite{Br:07, BrSc:08}.
This explicit approach does not lead to robust equilibrated estimator with respect to the coefficient jump
without introducing a constraint minimization (see \cite{CaZh:11}). The constraint minimization on each vertex 
patch may be efficiently solved by first computing an equilibrated flux and then calculating a divergence free
correction. For recent developments, see \cite{CaCaZh:20} and references therein. 
 

The purpose of this paper is to develop and analyze equilibrated a posteriori error estimators for discontinuous elements 
including both nonconforming  and discontinuous Galerkin elements. To do so, the first and the essential step
is to extend the Prager-Synge identity to be valid for piecewise $H^1(\Omega)$ functions. This will be done by establishing a 
generalized Prager-Synge inequality (see Theorem~3.1) that contains an additional term measuring the distance
between $H^1(\Omega)$ and piecewise $H^1(\Omega)$. 
Moreover, by using a Helmholtz decomposition, we will be able 
to show that the inequality becomes an identity in two dimensions (see Lemma~3.4).
A non-optimal inequality similar to ours was obtained earlier by
Braess, Fraunholz, and Hoppe in \cite{BFH:14} for the Poisson equation with pure Dirichlet boundary condition.
Based on the generalized Prager-Synge inequality and an equivalent form (see Corollary~3.2), 
the construction of an equilibrated a posteriori
error estimator for discontinuous finite element solutions is reduced to recover an equilibrated flux
in $H(\mbox{div};\Omega)$ and to recover 
either a potential function in $H^1(\Omega)$ or a curl free vector-valued
function in $H(\curll;\Omega)$.  

Recovery of equilibrated fluxes for discontinuous elements has been studied by many researchers. 
For discontinuous Garlerkin (DG) methods, equilibrated fluxes in Raviart-Thomas (RT) spaces were explicitly reconstructed
in \cite{Ai:07b} for linear elements and in \cite{ern2007accurate} for higher order elements. 
For nonconforming finite element methods, existing {\it explicit} equilibrated flux recoveries in RT spaces 
seem to be limited to the linear Crouzeix-Raviart (CR) and the quadratic Fortin-Soulie elements by Marini \cite{Ma:85}
(see \cite{Ai:05} in the context of estimator)
and Kim \cite{Ki:12}, respectively. 
For higher order nonconforming elements, a local reconstruction procedure was proposed by
Ainsworth and Rankin in \cite{Ai:08} through solving element-wise minimization problems. The
recovered flux is not in the RT spaces. Nevertheless,
the resulting estimator provides a guaranteed upper bound. 

In this paper, we will introduce a fully explicit post-processing procedure for recovering an equilibrated flux 
in the RT space of index $k-1$
for the nonconforming elements of odd order of $k\ge 1$. Currently, we are not able to extend our recovery 
technique to even orders. This is because our recovery procedure 
heavily depends on the finite element formulation and the properties of the nonconforming finite element space;
moreover, structure of the nonconforming finite element spaces of even and odd orders are fundamentally different.

Recovery of a potential function in $H^1(\Omega)$ for discontinuous elements was studied by some researchers
(see, e.g.,  \cite{Ai:08, Ai:07b, BFH:14}). Local approaches for recovering equilibrated flux in \cite{Br:07, BrSc:08, CaZh:11, BrPiSc:09, CaCaZh:20} may be directly applied (at least in two dimensions) for computing an approximation to the gradient in the curl-free space. 
(As mentioned previously, this approach requires solutions of local constraint minimization problems over vertex patches.) The resulting a posteriori error estimator from either the potential or the gradient
recoveries may be proved to be locally efficient. Nevertheless, to show independence of 
the efficiency constant on the jump, we have to assume that the distribution of the diffusion coefficient is 
quasi-monotone (see \cite{Pet:02}).

In this paper, we will employ a simple averaging technique over edges to 
recover a gradient in $H(\curll;\O)$. Due to the fact that the recovered gradient is 
not necessarily curl free, the reliability constant of the resulting estimator is no longer one. 
However, it turns out that the curl free constraint is not essential and, theoretically we are able to prove
that the resulting estimator has the robust local reliability as well as the robust local efficiency without the quad-monotone assumption. 
This is compatible with our recent result in 
\cite{CaHeZh:17} on the residual error estimator for discontinuous elements.

This paper is organized as follows. The diffusion problem and the finite element mesh are introduced in Section~2. 
The generalized Prager-Synge inequality for piecewise $H^1(\Omega)$ functions are established in Section~3. 
Explicit recoveries of an equilibrated flux and a gradient 
and the resulting a posteriori error estimator for discontinuous elements are described in Section~4.
Global reliability and local efficiency of the estimator are proved in Section~5. Finally, 
numerical results are presented in Section~6.

\section{Model problem}
\setcounter{equation}{0}

Let $\O$ be a bounded polygonal domain in $\mathbb{R}^d, d=2,3$, with Lipschitz 
boundary $\partial \O = \overline\Gamma_D \cup \overline \Gamma_N$, where
$ \overline \Gamma_D \cap \overline \Gamma_N = \emptyset$.
For simplicity, assume that
$\mbox{meas}_{d-1}(\Gamma_D) \neq 0$.
Considering the diffusion problem:
\begin{equation}\label{pde}
	-\gradt (A \nabla u)  =  f   \quad\mbox{in} \quad  \O,
\end{equation} 
with boundary conditions 
\[
	u = 0 \; \mbox{ on }  \Gamma_D \quad \mbox{and} \quad
	-A \nabla u \cdot \bn=g \;\mbox{ on } 
	\Gamma_N,
\]
where $\nabla \cdot$ and $\nabla$ are the respective divergence and gradient operators; $\bn$ is the outward unit 
vector normal to the boundary; $f \in L^2(\O)$ and $g\in H^{-1/2}(\Gamma_N)$ are given scalar-valued functions; and the diffusion coefficient $A(x)$ is symmetric, positive definite, and piecewise constant full tensor with respect to the domain
$\overline{\Omega} = \cup_{i=1}^n \overline{ \Omega}_i$.
Here we assume that the subdomain, $\Omega_i$ for $i=1, \cdots, n$, is open and polygonal.

 We use the standard notations and definitions for the Sobolev spaces. Let
\[
	H_D^1(\O) =\left\{v \in H^1(\O) \,:\,
	v=0 \mbox{ on } \Gamma_D \right\}.
\]
Then the corresponding variational problem of (\ref{pde}) is to  find $u \in H^1_D(\O)$ such that 
 \beq \label{vp}
	a(u,\,v):= (A\nabla u, \nabla v) = (f, v)- \left<g, v\right>_{\Gamma_N},
	\quad \forall  \;v\in H_D^1(\O),
 \eeq
where $(\cdot, \cdot)_{\omega}$ is the $L^2$ inner product on the domain $\o$. 
The subscript $\omega$ is omitted when $\o=\O$. 

\subsection{Triangulation}

 Let $\cT=\{K\}$ be a finite element partition of $\O$ that is regular, and denote 
 by $h_K$ the diameter of the element $K$. Furthermore, assume that the interfaces,
 \[
 	\Gamma = \{ \partial \O_i \cap \partial \O_j : i \neq j \mbox{ and }  i, j = 1,\cdots, n\},
 \]
 do not cut through any element $K \in \cT$.
 Denote the set of all edges of the triangulation $\cT$ by
  \[
 	\cE := \cE_I \cup \cE_D \cup \cE_N,
 \]
 where $\cE_I$ is the set of interior element edges, and $\cE_D$ and $\cE_N$ are the sets of 
boundary edges belonging to the respective $\Gamma_D$ and $\Gamma_N$.
  For each $F \in \cE$, denote by $h_F$ the length of $F$ and by
 $\bn_F$ a unit vector normal to $F$.
 Let $K_F^+$ and $K_F^-$ be the two elements sharing the common edge $F \in \cE_I$ 
 such that the unit outward normal of $K_F^-$ coincides with $\bn_F$. When $F \in \cE_D \cup \cE_N $,
 $\bn_F$ is the unit outward normal to $\partial \O$ and denote by $K_F^-$ the element having the edge $F$.
 
 
 \section{Generalized Prager-Synge inequality}\label{sec:3}
\setcounter{equation}{0}

For the conforming finite element approximation, the foundation of the equilibrated a posteriori error estimator is
the Prager-Synge identity \cite{PrSy:47}. That is, let $u\in H^1_D(\Omega)$ be the solution of (\ref{pde}), then
 \[
	 \|A^{1/2}\nabla\, (u-w)\|^2 
	 + \| A^{-1/2}\btau + A^{1/2}\nabla\, u\|^2 
	 = \| A^{-1/2}\btau + A^{1/2}\nabla\, w\|^2
 \]
for all $w\in H^1_D(\Omega)$ and for all $\btau \in \S_f(\O)$, where $\S_f(\O)$ is the so-called equilibrated flux space
defined by
 \[
\S_f(\O) = \Big\{ \btau \in H(\divvr;\O) : \gradt \btau =f  \mbox{ in } \O \;\mbox{ and }
  \; \btau \cdot \bn = g_{_{N}}  \Big\}.
  \]
Here, $H(\divvr;\O)\subset L^2(\O)^d$ denotes the space of all vector-valued functions whose divergence are in $L^2(\O)$.
The Prager-Synge identity immediately leads to
 \beq\label{P-S}
  \|A^{1/2}\nabla\, (u-w)\|^2 
   \leq   \inf_{\btau\in \S_f(\O)} \| A^{-1/2}\btau + A^{1/2}\nabla\, w\|^2.
 \eeq
Choosing $w\in H^1_D(\Omega)$ to be the conforming finite element approximation, then (\ref{P-S}) implies that
 \beq\label{eta_tau}
 \eta_{\tau}:=\| A^{-1/2}\btau + A^{1/2}\nabla\, w\|,\quad \forall\,\, \btau\in \S_f(\O)
 \eeq
is a reliable estimator with the reliability constant being one. 

We now proceed to establish a generalization of  (\ref{P-S}) for piecewise $H^1(\O)$ functions with applications to nonconforming and  discontinuous Galerkin finite element approximations. To this end, denote 
the broken $H^1(\O)$ space with respect to $\cT$ by
\[
	H^1(\cT)= \left\{ v \in L^2(\O) \,:\,
	v|_K \in H^1(K), \quad \forall \,K \in \cT
	  \right\}.
\]
Define  $\nabla_h$ be the discrete gradient operator on $H^1(\cT)$
such that for any $v \in H^1(\cT)$
\[
(\nabla_h v)|_K = \nabla (v|_K), \quad \forall K \in \cT.
\]

 \begin{thm}\label{Reliability}
Let $u\in H^1_{D}(\O)$ be the solution of  {\em (\ref{pde})}. In both two and three dimensions,
 for all $w \in H^1(\cT)$, we have
 \beq \label{implicit-reliability}
	 \|A^{1/2}\nabla_h(u-w)\|^2 \leq 
	 \inf_{\btau\in \S_f(\O)}\| A^{-1/2}\btau+
	 A^{1/2}\nabla_h w\|^2 + \inf_{v\in H^1_{D}(\O)}\|A^{1/2}\nabla_h(v-w)\|^2.
 \eeq
 \end{thm}

\begin{proof}
Let $w \in H^1(\cT)$, 
for all $\btau \in \S_f(\O)$ and for all $v\in H_{D}^1(\O)$, 
it follows from integration by parts and the Cauchy-Schwarz and Young's inequalities that
 \begin{eqnarray}\nonumber
  2\,(\nabla_h(u-w), A\nabla u+\btau)
  &=& 2\,(\nabla (u-v), A\nabla u+\btau) +2\,(\nabla_h(v-w), A\nabla u+\btau) \\[2mm] \nonumber
  &=& 2\, (\nabla_h(v-w), A\nabla u+\btau) \\[2mm] \label{ImRe:1}
  &\leq & \|A^{1/2}\nabla_h (v-w)\|^2 + \|A^{1/2}\nabla u + A^{-1/2} \btau\|^2 .
  \end{eqnarray}
It is easy to see that
 \[
 \|A^{1/2}\nabla_h w + A^{-1/2} \btau\|^2 
 = \|A^{1/2}\nabla_h (u-w)\|^2 +   \|A^{1/2}\nabla u 
 + A^{-1/2} \btau\|^2-2(\nabla_h(u-w), A\nabla u+\btau),
 \]
which, together with (\ref{ImRe:1}), implies
\begin{eqnarray*}
&& \|A^{1/2}\nabla_h (u-w)\|^2   \\[2mm]
 &=&\|A^{1/2}\nabla_h w + A^{-1/2} \btau\|^2-\|A^{1/2}\nabla u + A^{-1/2} \btau\|^2 
 + 2(\nabla_h(u-w), A\nabla u+\btau) \\[2mm]
 &\leq & \|A^{1/2}\nabla_h w + A^{-1/2} \btau\|^2+ \|A^{1/2}\nabla_h (v-w)\|^2
\end{eqnarray*}
for all  $\btau \in \S_f(\O)$ and all $v\in H_{D}^1(\O)$. This implies the validity of (\ref{implicit-reliability}) 
and, hence, the theorem.
\end{proof}

A suboptimal result for the Poisson equation ($A=I$) with pure Dirichlet boundary condition
is proved in \cite{BFH:14} by Braess, Fraunholz, and Hoppe:
\[
 \|\nabla_h(u-w)\| \leq  \inf_{\btau\in \S_f(\O)} \|\nabla w+ \btau\| + 2 \inf_{v\in H^1_0(\O)}\|\nabla_h(v-w)\|.
 \]

Let $H(\curll; \O )\subset L^2(\O)^d$ be the space of all vector-valued functions whose curl are in $L^2(\O)$, and denote its
curl free subspace by 
\[ 
	\mathring{H}_D(\curll; \O) =
	\left\{ \btau \in H(\curll; \O ): \curlt \btau = 0 \mbox{ in } \O \mbox{ and } \btau \cdot \bt = 0 \mbox{ on } \Gamma_D
	\right\},
	\]
where $\bt$ denotes  the tangent vector(s).

\begin{cor}\label{Reliability2}Let $u\in H^1_{D}(\O)$ be the solution of  {\em (\ref{pde})}. In both two and three dimensions,
for all $w \in H^1(\cT)$, we have
 \beq \label{implicit-reliability2}
	 \|A^{1/2}\nabla_h(u-w)\|^2 \leq 
	 \inf_{\btau\in \S_f(\O)}\| A^{-1/2}\btau+
	 A^{1/2}\nabla_h w\|^2 + \inf_{ \bgamma \in \mathring{H}_D(\curll; \O)} \|A^{1/2} (\bgamma - \nabla_h w)\|^2.
 \eeq
 \end{cor}
 
\begin{proof}
	The result of (\ref{implicit-reliability2}) is an immediate consequence of (\ref{implicit-reliability}) and the fact that
	$\nabla H^1_D(\Omega) = \mathring{H}_D(\curll; \O)$.
	\end{proof}

In the remaining section, we prove that, in two dimensions,  the inequality (\ref{implicit-reliability}) in Theorem \ref{Reliability} is indeed an equality.
For each $F \in \cE$, in two dimensions, assume that $\bn_F = (n_{_{1,F}}, n_{_{2,F}})$, then denote 
by $\bt_F = (-n_{_{2,F}}, n_{_{1,F}})$ the unit vector tangent to $F$ and by $\bs_\sF$ and $\be_\sF$ the start and end points of $F$, respectively, such that 
 $\be_\sF -\bs_\sF =h_F \bt_F$.
Let 
\[
	\cH = \left\{ v \in H^1(\O): \int_\O v \,dx = 0 \mbox{ and } \dfrac{\partial v}{\partial \bt} = 0
	 \mbox{ on } \Gamma_N \right\}.
\]
For a vector-valued function $\btau=(\tau_1, \tau_2) \in H(\curll; \O)$, define the curl operator by
\[
	\curlt \btau =\dfrac{\partial \tau_2}{\partial x} - \dfrac{\partial \tau_1}{\partial y}.
\]
For a scalar-valued function $v \in H^1(\O)$, define the formal adjoint operator of the curl by
\[
	\gperp v =\left(  \dfrac{\partial v}{ \partial y}, \,- \dfrac{\partial v}{ \partial x}\right) .
\]

For a fixed $w \in H^1(\cT)$, there exist unique $\phi \in H_D^1(\O)$ and $\psi \in \cH$ for the following 
Helmholtz decomposition (see, e.g.,  \cite{Ai:08}) such that
\beq \label{HM:0}
	A\nabla_h (u-w) = A \nabla \phi + \gperp \psi,
\eeq
and $\phi$ and $\psi$ satisfy 
\[
	(A \nabla \phi, \nabla v) = (A \nabla_h(u-w), \nabla v) \quad \forall v\in H_D^1(\O),
\]
and
\[
	(A^{-1} \gperp \psi, \gperp w) = (\nabla_h (u-w), \gperp w) \quad \forall w\in  \cH,
\]
respectively.
It is easy to see that 
$\nabla \phi$ and $\gperp \psi$ are orthogonal with respect to the $L^2$ inner product, which  yields
\beq\label{HM}
	\|A^{1/2} \nabla_h (u - w )\|^2 = 
	\|A^{1/2} \nabla \phi\|^2 + \|A^{-1/2} \gperp \psi\|^2.
\eeq

\begin{lem}\label{7.2}
Let $w$ be a fixed function in $H^1(\cT)$ and  $\phi$ and $\psi$ be the corresponding Helmholtz decomposition of $w$ given in {\em(\ref{HM:0})}.
	We have
	\beq \label{HM:1}
		  \inf_{\btau\in \S_f(\O)} \!\!\| A^{-1/2}\btau+ A^{1/2}\nabla_h w\|
		  =\|A^{1/2} \nabla \phi\|
		  \mbox{ and }
		  \inf_{v\in H^1_{D}(\O)} \!\! \|A^{1/2}\nabla_h(v-w)\| = \|A^{-1/2} \gperp \psi\|.
		  \eeq
\end{lem}

\begin{proof}
	For any $\btau \in \S_f(\O)$, (\ref{HM:0}) and integration by parts give
	\[
	\|A^{1/2} \nabla \phi\|^2
	=(A \nabla_h (u - w), \nabla \phi) = 
	  (A \nabla u + \btau, \nabla \phi)-(\btau + A \nabla_h w, \nabla \phi)
	  =-(\btau + A \nabla_h w, \nabla \phi),
   \]
	which, together with the Cauchy-Schwarz inequality and the choice 
	$\btau = \gperp \psi -A \nabla u \in \S_f(\O)$, yields the first equality in (\ref{HM:1}) as follows:
	\[
	\|A^{1/2} \nabla \phi\| 
	\le  \inf_{\btau\in \S_f(\O)}\| A^{-1/2}\btau+ A^{1/2}\nabla_h w\| 
	 \leq \|A^{1/2}\nabla_h(u-w) - A^{-1/2}\gperp \psi\|
	 =\|A^{1/2} \nabla \phi\| .
	\]
	 	 
	Now we proceed to prove the second equality in (\ref{HM:1}). For any $v \in H_D^1(\O)$, by (\ref{HM:0}) and integration by parts, we have
	\[
	 	\|A^{-1/2} \gperp \psi\|^2= (\nabla_h(u-w) , \gperp \psi) = 
		(\nabla_h (v - w), \gperp \psi).
	 \]
The second equality in (\ref{HM:1}) is then a consequence of the Cauchy-Schwartz inequality 
and the choice of $v = u - \phi \in H_D^1(\O)$:
	 \[
	 	\|A^{-1/2} \gperp \psi\|
		 \le \inf_{v\in H^1_{D}(\O)}\|A^{1/2}\nabla_h(v-w)\|
		 \leq \|A^{1/2}\nabla_h(u-\phi-w\| = \|A^{-1/2} \gperp \psi\|.
	 \]
	This completes the proof of the lemma.
\end{proof}

 \begin{lem}\label{PS:2d}
Let $u\in H^1_{D}(\O)$ be the solution of  {\em (\ref{pde})}. In two dimensions,
 for all $w \in H^1(\cT)$, we have
 \beq \label{equality}
	 \|A^{1/2}\nabla_h(u-w)\|^2 =
	 \inf_{\btau\in \S_f(\O)}\| A^{-1/2}\btau+
	 A^{1/2}\nabla_h w\|^2 + \inf_{v\in H^1_{D}(\O)}\|A^{1/2}\nabla_h(v-w)\|^2.
 \eeq
 \end{lem}
 
 \begin{proof}
 The identity (\ref{equality}) is a direct consequence of (\ref{HM}) and Lemma~\ref{7.2}.
 \end{proof}
 
 \begin{rem}
 	It is easy to see that if $w \in H_D^1(\O)$ in \upshape{Lemma \ref{PS:2d}}, i.e., $w$ is conforming, the second part on the right of (\ref{equality}) vanishes. It is thus natural to refer $ \inf \limits_{\btau\in \S_f(\O)}\| A^{-1/2}\btau+
	 A^{1/2}\nabla_h w\|^2$ or $\|A^{1/2} \nabla \phi\|$ as the conforming error and $\inf \limits_{v\in H^1_{D}(\O)}\|A^{1/2}\nabla_h(v-w)\|^2$ or  $\|A^{-1/2} \gperp \psi\|$ as the nonconforming error. 
 \end{rem}
 
 For each $K \in \cT$, denote by $\Lambda_K$ and $\lambda_K$ the maximal and
minimal eigenvalues of $A_K = A|_K$, respectively. For each $F \in \cE$, let $\Lambda_F^\pm = \Lambda_{K_F^\pm}$,
$\lambda_{\sF}^\pm =\lambda_{K_F^\pm}$,  and  $\lambda_F = \min\{\lambda_F^+, \lambda_F^-\}$
if $F \in \cE_I$ and $\lambda_F = \lambda_F^-$ if $F \in \cE_D \cup \cE_N$.  
 To this end, let
\[
	\Lambda_\cT = \max_{K \in \cT} \Lambda_K \quad \mbox{and} \quad
	\lambda_\cT=  \min_{K \in \cT} \lambda_K.
\]
Assume that each local matrix $A_K$ is similar to the identity matrix in the sense that its 
	maximal and minimal eigenvalues are almost of the same size. More precisely,
	there exists a moderate size constant $\kappa > 0$ such that
	\[
		\dfrac{\Lambda_K}{ \lambda_K} \le \kappa , \quad \forall \,K \in \cT.
	\] 
Nevertheless, the ratio of global maximal and minimal eigenvalues, $\Lambda_\cT/\lambda_\cT$, 
is allowed to be very large.

For a function $w \in H^1(\cT)$ , denote its traces on $F$ by $w|_F^-:= (w|_{K_F^-})|_F $ and $w|_F^+:= (w|_{K_F^+})|_F$
 and the jump of $w$ across the edge $F$ by 
\[
	\jump{w}|_F = \left\{
	\begin{array}{ll}
		w|_{F}^- - w|_{F}^+,&\forall \,F \in \cE_I,\\[2mm]
		w|_F^-, & \forall \, F \in \cE_D \cup \cE_N.
	\end{array}
	\right.
\]

 In the following lemma, we show the relationship between the nonconforming error and the residual based error of solution jump on edges. It is noted that the constant is robust with respect to the coefficient jump.

 \begin{lem} \label{lem-nc-error} Let $w$ be a fixed function in $H^1(\cT)$.
In two dimensions, there exists a constant $C_r$ that is independent of the jump of the coefficient such that
\beq \label{rel:nc-b}
	\inf_{\btau \in \mathring{H}_D(\curll; \O)} \|A^{1/2} ( \btau - \nabla_h  w)\|
	\le C_r\, \left( \sum_{F \in \cE_I \cup \cE_D} \lambda_F h_F^{-1} \|\jump{w}\|_{0,F}^2 \right)^{1/2}.
\eeq
\end{lem}

\begin{proof}
Let $\psi$ be given in  the Helmholtz decomposition in (\ref{HM:0}), then 
integration by parts gives
\[
	\|A^{-1/2} \gperp \psi \|^2 =(\nabla_h (u-w), \gperp \psi) = -
	\sum_{F \in \cE_I \cup \cE_D} \int_F \jump{w}\, \left(\gperp \psi \cdot \bn_F\right) \,ds .
\]
Without loss of generality, assume that $\lambda_F^- \le \lambda_F^+$ for each $F \in \cE_I$.
It follows from Lemma 2.4 in \cite{CaHeZh:17} and the Cauchy-Schwarz inequality that
\begin{eqnarray*}
\sum_{F \in \cE_I\cup \cE_D} \int_F \jump{w}\, \left(\gperp \psi \cdot \bn_F\right) \,ds
&\le &C \sum_{F \in \cE_I\cup \cE_D}h_F^{-1/2} \| \jump{w}\|_{0,F} \|\gperp \psi\|_{0,K_F^-}\\[2mm]
&\le& C \left( \sum_{F \in \cE_I\cup\cE_D} \lambda_F h_F^{-1}  \|\jump{w}\|_{0,F}^2\right)^{1/2} \|A^{-1/2}\gperp \psi\|,
\end{eqnarray*}
which, together with the above equality, yields
 \[
\|A^{-1/2} \gperp \psi \|
 \le C\, \left( \sum_{F \in \cE_I \cup \cE_D} \lambda_F h_F^{-1} \|\jump{w}\|_{0,F}^2 \right)^{1/2}.
\]
This completes the proof of the lemma.
\end{proof}

\section{Error estimators and indicators}
\setcounter{equation}{0}

\subsection{NC finite element approximation}
For the convenience of readers, in this subsection we introduce the nonconforming finite element space and its properties.


Let $\mathbb{P}_k(K)$  and $\mathbb{P}_k(F)$ be the spaces of polynomials of degree less than or equal to $k$
on the element $K$ and $F$, respectively.
Define the nonconforming finite element space of order $k (k\ge 1)$ on the triangulation $\cT$ by
\beq 	\cU^k(\cT) \! = \!
  \left\{ v \in L^2(\O) \! :  v|_K \in \mathbb{P}_k(K), \forall \, K \in \cT\, \mbox{and}\,
	 \int_\sF \jump{v}\, p \,ds =0, \forall \, p \in \mathbb{P}_{k-1}(F),  \;\forall \,  F \in \cE_I  \right\}
\eeq	
and its subspace by
\[
	\cU_D^k(\cT) =
	\left\{
	v \in \cU^k(\cT) \! :\, \int_\sF v \, p \,ds =0, \quad \forall \,p \in \mathbb{P}_{k-1}(F) \; \mbox{and} \; \forall \,F \in \cE_D
	\right\}.
\]
The spaces defined above are exactly the same as those defined in \cite{CrRa:73}
for $k=1$, \cite{FoSu:83} for $k=2$, \cite{ChaLee:00} for $k=4$ and $6$, \cite{Ai:08} for general odd order, and \cite{Sto:06, Ba:07} for general order.
Then the nonconforming finite element approximation of order $k$ is to find $u_{_\cT} \in \cU_D^{k}(\cT)$ such that
\beq \label{NC solution}
	a_h(u_{_\cT}, v) := (A \nabla_h u_{_{\cT}}, \nabla_h v)
	=(f,v)-\left<g, v\right>_{\Gamma_N}, \quad \forall \, v  \in \cU_\sD^{k}(\cT).
\eeq

Below we describe basis functions of $\cU^{k}(\cT)$ and their properties.
To this end, for each $K\in\cT$, let $m_k=\mbox{dim}(\mathbb{P}_{k-3}(K))$ for $k> 3$ and $m_k=0$ for $k\le 3$. 
Denote by
$\{\bx_j, j=1, \cdots, m_k\}$
the set of all interior Lagrange points in $K$ with respect to
the space $\mathbb{P}_k(K)$ and
by $P_{j,K} \in \mathbb{P}_{k-3}(K)$ the nodal basis function corresponding to $\bx_j$, i.e., 
 \[
 P_{j,K}(\bx_i) =\delta_{ij} \,\,\mbox{ for } i=1,\,\cdots,\, m_k,
 \]  
where $\delta_{ij}$ is the Kronecker delta function.
For each $0 \le j \le k-1$, let
$L_{j,F}$ be the $j$th order Gauss-Legendre polynomial on $F$ such that $L_{j,F}(\be_\sF)=1$. 
Note that $L_{j,F}$ is an odd or even function when $j$ is odd or even. Hence, $L_{j,F}(\bs_F)=-1$ for odd $j$  and $L_{j,F}(\bs_F)=1$ for even $j$. 

For odd $k$, the set of degrees of freedom of $\cU^k(\cT)$ (see Lemma 2.1 in \cite{Ai:08}) can be given by
\begin{equation}
	 \int_K v \,P_{j,K} \,dx, \quad j=1, \,\cdots, \,m_k
\end{equation} 
for all $K \in \cT$ and
\begin{equation}
	\int_F v \,L{_{j,F}} \,ds, \quad  j=0, \,\cdots,\, k-1
\end{equation}
for all $F \in \cE$.
Define the basis function $\phi_{i,K} \in \cU^k(\cT)$ satisfying 
\begin{equation} \label{element-basis}
\left\{
	\begin{array}{lll}
	\int_{K'} \phi_{i,K} \,P_{j,K'} \,dx =  \delta_{ij}\delta_{KK'},  & \forall \,j=1,\, \cdots ,\,m_k, & \forall \, K' \in \cT,\\[4mm]
	\int_{F} \phi_{i,K} \,L_{j,F} \,ds=0, & \forall \,j=0,\, \cdots,\, k-1, & \forall \,F \in \cE,
	\end{array}
	\right.
\end{equation}
for $i=1, \,\cdots, \, m_k$ and $ K\in \cT$, 
and the basis function $\phi_{i,F} \in \cU^k(\cT)$ satisfying
\begin{equation}\label{edge-basis}
\left\{
	\begin{array}{lll}
	\int_{K} \phi_{i,F} \,P_{j,K} \,dx=0, & \forall \,j=1, \,\cdots,\,m_k, & \forall \,K \in \cT, \\[4mm]
	\int_{F'} \phi_{i,F} \,L_{j,F'} \,ds = \delta_{ij} \delta_{_{FF'}},  & \forall \,j=0,\, \cdots,\, k-1, & \forall \,F' \in \cE,
	\end{array}
	\right.
\end{equation}
for $i=0, \,\cdots, \, k-1$ and $ F \in \cE$.
Then the nonconforming finite element space is the space spanned by all these basis functions, i.e.,
\[
	\cU^k(\cT) = \mbox{span}
	\left\{\phi_{i,K}:\, K \in \cT\right\}_{i=1}^{m_k} \oplus
	\mbox{span} \left\{\phi_{i,F}: \; F \in \cE \right\}_{i=0}^{k-1}.
\] 


\begin{lem}\label{element-basis-0-boundary}
For all $K \in \cT$, the basis functions $\{ \phi_{j,K}\}_{j=1}^{m_k}$ have support on $K$ and 
vanish on the boundary of $K$, i.e.,
\[
	\phi_{j,K}\equiv 0 \; \mbox{ on }\partial K.
\]
\end{lem}

\begin{proof}
Obviously, (\ref{element-basis}) implies that $\mbox{support}\{\phi_{j,K}\} \in \overline{K}$.
To show that $\phi_{j,K}|_{\partial K}  \equiv 0$, considering each edge $F \in \cE_K$,
 the second equation of (\ref{element-basis}) indicates that there exists 
$a_\sF \in R$ such that
\[
	\phi_{j,K}|_F = a_\sF L_{k,F}.
\]
Note that $L_{k,F}$ is an odd function on $F$ and that values of $L_{k,F}$ at two end-points of $F$ are 
$-1$ and $1$, respectively.
Now the continuity of $\phi_{j,K}$ in $K$ implies that $a_\sF =0$ and, hence, $\phi_{j,K} \equiv 0$ on $\partial K$.
\end{proof}

For each $K$, denote by $\cE_K$ the set of all edges of $K$.
For each $F \in \cE$, denote by $\o_F$ the union of all elements that share the common edge $F$;
and define a sign function $\chi_F$ on the set $\cE_{K_F^+} \cup \cE_{K_F^-} \setminus \{F\}$ 
(when $F$ is a boundary edge, let $\cE_{K_F^+} =\emptyset$)  such that
\[
	\chi_{F}(F') = \left\{
	\begin{array}{lll}
		1, & \mbox{if} &\be_{_{F'}} =\bar{F} \cap \bar{F}^\prime, \\[2mm]
		-1, & \mbox{if} &\bs_{_{F'}} =\bar{F} \cap \bar{F}^\prime.
	\end{array}
	\right.
\]

\begin{lem}\label{edge-a}
For all $F \in \cE$, the basis functions $\{\phi_{j,F}\}_{j=0}^{k-1}$ have support on $\overline{\o}_F$, and their restrictions
on $ \cE_{K_F^+} \cup \cE_{K_F^-}$
has the following representation:
\begin{equation}\label{edge-1-a}
	\phi_{j,F} = \left\{
	\begin{array}{lll}
	\dfrac{1}{\|L_{j,F}\|_{0,F}^2} \left( L_{j,F} - L_{k,F} \right), & \mbox{on} &F,\\[4mm]
	0,  & \mbox{on} &\cE_{K_F^+} \cup \cE_{K_F^-} \setminus \{F\}
	\end{array}
	\right.
\end{equation}
when $j$ is odd, and
\begin{equation}\label{edge-1-b}
	\phi_{j,F} = \left\{ 
	\begin{array}{lll}
		\dfrac{1}{\|L_{j,F}\|_{0,F}^2} L_{j,F},  &\mbox{on}& F,\\[6mm]
		\dfrac{  \chi_{F}(F')}{\|L_{j,F}\|_{0,F}^2} L_{k,F'},
		&\mbox{on}&F' \in  \cE_{K_F^+} \cup \cE_{K_F^-} \setminus \{F\}\\[4mm]
	\end{array}
	\right.
\end{equation}
when $j$ is even.

\end{lem}

\begin{proof}
By (\ref{edge-basis}), it is easy to see that
support of $\phi_{j,F}$ is $\overline{\o}_F$.
	Since $\phi_{j,F}|_F^{\pm} \in \mathbb{P}_k(F)$, 
	there exist constants $a_{i,F}^{\pm}$ such that
	\[
	\phi_{j,F}|_\sF^\pm = \sum_{i=0}^{k} a_{{i,F}}^\pm \,L_{i,F}.
	\] 
	Using (\ref{edge-basis}) and the orthogonality of $\{L_{i,F}\}_{i=0}^{k}$, it is obvious that
	\[
	a_{i,F}^\pm= \left\{
	\begin{array}{lll}
		\|L_{j,F}\|_{{0,F}}^{-2}, &\mbox{for} \,i = j, \\[2mm]
		0, &  \mbox{for} \,0\le i\le k-1\mbox{ and }i\neq j 
	\end{array}
	\right.
	\]
and, hence,
	\beq \label{edge-basis-1-c}
		\phi_{j,F}|_F^\pm = \dfrac{1}{\|L_{j,F}\|_{0,F}^2} L_{j,F} + a_{{k,F}}^\pm L_{k,F}.
	\eeq
	By (\ref{edge-basis}), it is also easy to see
	that there exists constant $a_{{j,F,F'}}$ for each $F' \in \cE_{K_F^+} \cup \cE_{K_F^-} \setminus \{F\}$ such that
	\beq\label{edge-basis-1-e}
		\phi_{j,F}|_{F'} = a_{{j,F,F'}} L_{{k,F'}}.
	\eeq
	Since $L_{k,F'}$ is an odd function for all $F' \in \cE_{K_F^+} \cup \cE_{K_F^-} \setminus \{F\}$
	and $\phi_{j,F}$ is continuous in $K_F^+$ and $K_F^-$, (\ref{edge-basis-1-e}) implies that  
	\beq\label{edge-basis-1-d}
		\phi_{j,F}|_K(\bs_\sF) = \phi_{j,F}|_{K}(\be_\sF), \quad \, K \in \{K_F^+, K_F^-\}.
	\eeq
	Combining the facts that $L_{j,F}(\be_\sF)=-L_{j,F}(\bs_\sF) =1 $ for odd $j$
	and that $L_{j,F}(\be_\sF)=L_{j,F}(\bs_\sF) =1 $ for even $j$, (\ref{edge-basis-1-c}), and 
	(\ref{edge-basis-1-d}), we have 
	\[ 
		a_{{k,F}}^\pm=
		\left\{ 
		\begin{array}{lll}
		-\dfrac{1}{\|L_{j,F}\|_{0,F}^2}, & \mbox{for odd } j,\\[4mm]
		0,&\mbox{for even } j,
		\end{array}
		\right.
	\]
which, together with (\ref{edge-basis-1-c}), leads to the formulas of $\phi_{j,F}|_{F}$ in (\ref{edge-1-a})
and (\ref{edge-1-b}).
Finally, for each $F' \in \cE_{K_F^+} \cup \cE_{K_F^-} \setminus \{F\}$, $a_{_{j,F,F'}}$ in  (\ref{edge-basis-1-e})  can be directly computed based on the continuity of $\phi_{j,F}$ in $K_F^+$ and $K_F^-$.
	This completes the proof of the lemma.
	
\end{proof}

\begin{rem}
As a consequence of {\em Lemma~\ref{edge-a}}, 
the basis function $\phi_{j,F}$ is continuous on the edge $F$, i.e.,
$\jump{\phi_{j,F}}\big|_F =0$ for all $j=0,\,\cdots,\,k-1$;
moreover, $\phi_{j,F}$ vanishes at end points of $F$, i.e.,
$\phi_{j,F}(\bs_F) = \phi_{j,F}(\be_F) = 0$, for odd $j$.
\end{rem}

\begin{lem}\label{edge-1}
Let $F$ be an edge of $K$. Assume that $p \in \mathbb{P}_{k-1}(F)$.
Then we have that
\beq\label{edge-c}
	\int_{\partial K} p \,\phi_{j,F} \,ds = \int_F p \,\phi_{j,F} \,ds.
\eeq
Moreover, if 
$
	\displaystyle \int_F p \, \phi_{j,F} \,ds =0$ for all  $j=0, \cdots, k-1$,
 then $p \equiv 0$ on $F$.
 \end{lem}

 \begin{proof}
 Since $\{L_{j,F}\}_{j=0}^{k}$ are orthogonal polynomials on $F$, Lemma \ref{edge-1} is 
 a direct consequence of Lemma \ref{edge-a}.
 \end{proof}

\subsection{Equilibrated flux recovery}\label{odd-flux-recover}
In this subsection, we introduce a fully explicit post-processing procedure for recovering an equilibrated flux.
To this end, 
define $f_{k-1} \in L^2(\O)$ by 
\[
	f_{k-1}|_K = \Pi_{K} (f), \quad \forall \,K \in \cT,
\]
where  $\Pi_{K}$ 
is the $L^2$ projection
onto $\mathbb{P}_{k-1}(K)$.
For simplicity, assume that the Neumann data $g$ is a piecewise polynomial of degree less than or equal to $k-1$,
i.e., $g|_F\in \mathbb{P}_{k-1}(F)$ for all $F \in \cE_N$.

Denote the $H(\mbox{div}; \O)$ conforming Raviart-Thomas (RT) space of index $k-1$ with respect to $\cT$
by
\[
RT^{k-1}(\cT) = \left\{ \btau \in H(\mbox{div};\O) \,: \, \btau|_K \in RT^{k-1}(K),  \;\forall\, K\in\cT \right\},
\]
where $RT^{k-1}(K) = \mathbb{P}_{k-1}(K)^d + \bx \, \mathbb{P}_{k-1}(K) $.
Let
\[
 \Sigma_f^{k-1}(\cT) = \left\{ \btau \in RT^{k-1}: \gradt \btau =f_{k-1} \,\mbox{in} \, \O \quad\mbox{and}\quad \btau \cdot \bn_F=
  g \, \mbox{on} \,\Gamma_N \right\}.
\]   
On a triangular element $K \in \cT$, a vector-valued function $\btau$ in $RT^{k-1}(K)$  is characterized 
by the following degrees of freedom (see Proposition 2.3.4 in \cite{BoBrFo:13}):
 \[
 	\int_K \btau \cdot \bzeta \,dx , \quad \forall \, \bzeta \in \mathbb{P}_{k-2}(K)^d,	
 \]
 and
 \[	
 	\int_F ( \btau \cdot \bn_F) \,p \,ds, \quad \forall \,p \in \mathbb{P}_{k-1}(F) \mbox{ and } \; \forall \, F \in \cE_K.
 \]
 
 For each $K \in \cT$, define a sign function $\mu_K$ on $\cE_K $ such that
\[
\mu_K(F) = 
\left\{
\begin{array}{lll}
	1, &\mbox{if} & \bn_K|_F = \bn_F,\\[2mm]
	-1, &\mbox{if} & \bn_K|_F = -\bn_F.
\end{array}
\right.
\]
Define the numerical flux 
 \beq\label{num-flux}
 \tilde\bsigma_{_\cT} = -A \nabla_h u_{_\cT} \quad \mbox{and} \quad
 \tilde \bsigma_K = -A \nabla (u_{_\cT}|_K), \quad \forall \,K\in \cT.
 \eeq
With the numerical flux $\tilde \bsigma_{_{\cT}}$ given in (\ref{num-flux}), for each element $K \in \cT$, we recover a flux
$\hat \bsigma_K \in RT^{k-1}(K)$ such that:
 \begin{equation}\label{rt:flux:b}
 	\int_K \hat\bsigma_K \cdot \btau\,dx = \int_{K}\tilde\bsigma_{_\cT} \cdot \btau \,dx, \quad
	\forall  \,\btau \in \mathbb{P}_{k-2}(K)^d
 \end{equation}
 and that
 \beq \label{rt:flux:a}
	\!\!  \int_{F}\! \hat \bsigma_K \cdot \bn_F \,L_{i,F} \,ds
	=\!\!\left\{
	\begin{array}{lll}
	 	\!\! \mu_K(F) \| L_{i,F}\|_{0,F}^2 
	\left( \mathlarger\int_{K} \tilde\bsigma_{_\cT} \cdot \nabla \phi_{i, F} \,dx+
	\mathlarger\int_K f \,\phi_{i,F} \,dx \right), &  \forall \, F \in \cE_K \setminus \cE_N,
	\\[6mm]
	 \!\!	\mu_K(F) \| L_{i,F}\|_{0,F}^2 
	\left(\mathlarger\int_{\sF} g \,\phi_{i,F} \,ds \right), & \forall \, F \in \cE_K \cap \cE_N
	\end{array}
	\right.
 \eeq
 for $i=0,\, \cdots, \, k-1$.
 Now the global recovered flux $\hat \bsigma_{_\cT}$ is defined by
 \beq\label{global-flux}
 	\hat \bsigma_{_\cT}\big|_K =\hat \bsigma_{K}, \quad\forall\,\, K\in \cT.
 \eeq

\begin{lem}\label{equi:d}
	Let $u_{_{\cT}}$ be the finite element solution in \em{(\ref{NC solution})}
 and $\hat\bsigma_{_\cT}$ be the recovered flux defined in (\ref{global-flux}). 
 Then for any $K \in \cT$, the following equality
 \begin{equation}\label{equi:b}
	\int_{\partial K} \hat \bsigma_{_\cT} \cdot \bn_K \,q\,dx =
	\int_K \tilde \bsigma_{_\cT} \cdot  \nabla q \,dx+ \int_K f \,q\,dx
 \end{equation}
 holds for all $q \in \mathbb{P}_{k}(K)$.
\end{lem}
\begin{proof}
Without loss of generality, assume that $K \in \cT$ is an interior element.
For each $q \in \mathbb{P}_k(K)$, there exist 
$a_{{j,F}} $  and 
$a_{{j,K}} $
 such that
\[
	q =\sum_{F \in \cE_K} \sum_{j=0}^{k-1} a_{{j,F}} \,\phi_{j,F} 
	+ \sum_{j=1}^{m_k} a_{{j,K}} \,\phi_{j,K} \equiv \sum_{F \in \cE_K} q_F + q_K.
\] 
 It follows from Lemma \ref{element-basis-0-boundary}, (\ref{edge-c}), Lemma \ref{edge-a}, 
 and the definition of the recovered flux $\hat \bsigma_{_\cT}$ in (\ref{rt:flux:a}) that
\begin{eqnarray}\label{elem-2}
	&& \int_{\partial K} \hat \bsigma_K \cdot \bn_K \,q \,ds
	=\sum_{F\in \cE_K} \sum_{j=0}^{k-1} a_{{j,F}}\int_{F}  \hat \bsigma_K \cdot \bn_K \, \phi_{j,F} \,ds
	\nonumber\\[2mm]
	&=&\sum_{F\in \cE_K} \sum_{j=0}^{k-1}  \dfrac{a_{{j,F}} \,\mu_K(F)}{\|L_{j,F}\|_{F}^2}
		\int_{F} \hat \bsigma_K \cdot \bn_F \, L_{j,F} \,ds   
	=\sum_{F\in \cE_K} \sum_{j=0}^{k-1}
	a_{{j,F}}\left(\int_K \tilde \bsigma_{_\cT} \cdot \nabla \phi_{j,F} \,dx+\int_K f\, \phi_{j,F} \,dx \right)
	 \nonumber\\[2mm]
	 &=&
	 \sum_{F \in \cE_K}
	 \left(\int_K \tilde \bsigma_{_\cT} \cdot \nabla q_F \,dx+\int_K f\, q_F \,dx \right).
\end{eqnarray}
Choosing $v=\phi_{j,K}$ in  (\ref{NC solution}) gives
 \[
	\int_K  \tilde \bsigma_{_\cT}  \cdot \nabla \phi_{j,K} \,dx + \int_K f\, \phi_{j,K} \,dx =0
 \]
  for $j=1, \, \cdots,\,m_k$.
  Multiplying the above equality by $a_{{j,K}}$ and summing over $j$ imply
   \begin{equation} \label{elem-0}
	\int_K  \tilde \bsigma_{_\cT}  \cdot \nabla q_K \,dx + \int_K f\, q_K \,dx =0.
 \end{equation}
Now (\ref{equi:b}) is the summation of (\ref{elem-2}) and (\ref{elem-0}).
This completes the proof of the lemma.
\end{proof}

\begin{thm}
Let $u_{_{\cT}}$ be the finite element solution in {\em(\ref{NC solution})}.
Then the recovered flux $\hat \bsigma_{_\cT}$ defined in {\em(\ref{global-flux})}
 belongs to $\S_{f}^{k-1}(\cT)$.

\end{thm}
\begin{proof}
First we prove that $\hat \bsigma_{_\cT} \in H(\mbox{div}; \O)$.
For each $F\in \cE_I$, note that $\hat \bsigma_{_\cT}|_\sF^\pm \in \mathbb{P}_{k-1}(F)$.
	 Then it follows from Lemma \ref{edge-a}, (\ref{rt:flux:a}), the assumption that $g|_F\in \mathbb{P}_{k-1}(F)$,
	 and (\ref{NC solution}) with $v= \phi_{j,F}$ that
	 \begin{eqnarray*}
	 	\int_\sF \jump{\hat \bsigma \cdot \bn_F}\, \phi_{j,F} \,ds &=&
		 \sum_{K \in \{K_F^+, K_F^-\}} \dfrac{\mu_K(F)}{\| L_{k,F}\|_\sF^2} 
		 \int_{F}\hat \bsigma_K \cdot \bn_F\, L_{j,F} \,ds \\[2mm]
		&=&
		 \sum_{K \in \{K_F^+, K_F^-\}} \left( \int_{K }\tilde \bsigma_{_\cT} \cdot \nabla \phi_{j,F} \,ds 
		 + \int_{K } f \, \phi_{j,F} \,ds \right)\\[2mm]
		 &=&
		  \int_{\o_F }\tilde \bsigma_{_\cT} \cdot \nabla \phi_{j,F} \,ds 
		 + \int_{\o_F } f \, \phi_{j,F} \,ds  - \int_{\Gamma_N \cap \partial \o_F} g \, \phi_{j,F} \,ds \\[2mm]
		 &=&0
	 \end{eqnarray*}
	 for $j=0, \,\cdots, \,k-1$. Now Lemma \ref{edge-1} implies that
	 $\jump{\hat \bsigma_{_\cT} \cdot \bn_F}|_\sF =0$ and, hence, $\hat \bsigma_{_\cT} \in H(\mbox{div} ,\O)$.
	Second, for each $K \in \cT$ and for any $p \in \mathbb{P}_{k-1}(K)$, note that $\nabla p \in \mathbb{P}_{k-2}(K)^d$.
	By integration by parts, (\ref{rt:flux:b}), and Lemma \ref{equi:d}, we have
	\begin{eqnarray*}
		\int_K \nabla \cdot \hat \bsigma_K \,p \,dx &=&
		-\int_K   \hat \bsigma_K \cdot \nabla p \,dx
		+\int_{\partial K} \hat \bsigma_K \cdot \bn_K \, p \,ds\\[2mm]
		&=&-\int_K \tilde \bsigma_{_{\cT}}  \cdot \nabla p \,dx
		  + \left( \int_K \tilde \bsigma_{_{\cT}} \cdot \nabla p\,dx+ \int_K f \,p \,dx \right)
		  =\int_K f \, p \,dx,
	\end{eqnarray*}
	which implies that 
	$\nabla  \cdot \hat \bsigma_{_\cT} = f_{k-1}$ in $\O$.
	
Finally, for $F \in \cE_N$, Lemma \ref{edge-1} and (\ref{rt:flux:a}) gives
\[
	\int_\sF  \hat \bsigma_{_\cT} \cdot \bn_F \phi_{j,F} \,ds=
	\|L_{j,F}\|^{-2}_{0,F}\int_\sF  \hat \bsigma_{_\cT} \cdot \bn_F L_{j,F} \,ds
	=\int_\sF g\,\phi_{j,F} \,ds,
\]
for $j=0, \,\cdots, \,k-1$, which, together with Lemma \ref{edge-1}, implies that
$ \hat \bsigma_{_\cT} \cdot \bn_F =g|_F$ for all $F\in \cE_N$.
This completes the proof of the theorem.
\end{proof}

\subsection{Gradient recovery}\label{gradient recovery}
In this subsection, we recover a gradient in the space of $H(\curll; \O)$ for the nonconforming finite element solutions of odd orders in the two dimensions. We note that such recovery is fully explicit through a simple weighted average on each edge. Such recovery technique can be easily extended to three dimensional finite element problems with the average on facets.  
 For the first order nonconforming Crouzeix-Raviart element, the weighted average approach is first introduced in \cite{CaZh:10a}.
 Define
 \[
 	H_D(\curll; \O) = \{ \btau \in H(\curll;\O): \btau \cdot \bt =0 \mbox{ on } \Gamma_N.\}
 \]
To this end, denote the $H_D(\curll; \O)$ conforming N\'ed\'elec (NE) space of index $k-1$ with respect to $\cT$ by
 \[
 N\!E^{k-1}(\cT) =\left\{\btau\in H_D(\curll; \O)\, :\, \btau|_K\in N\!E^{k-1}(K), \; \forall\, K\in\cT\right\},
 \]
where  $N\!E^{k-1}(K) =\mathbb{P}_{k-1}(K)^2+ (-y, x) \,\mathbb{P}_{k-1}(K)$.
On a triangular element $K \in \cT$, a vector valued function $\btau \in N\!E^{k-1}(K)$ is characterized by the 
following degrees of freedom (see Proposition 2.3.1 in \cite{BoBrFo:13}):
\[
		\int_K \btau \cdot \bzeta \,dx,  \quad \forall \,\bzeta \in \mathbb{P}_{k-2}(K)^2
		\quad\mbox{and}\quad
  \int_{F} \big( \btau \cdot \bt \big) \,p \,dx,   \quad
		 \forall \, p \in \mathbb{P}_{k-1}( F)\mbox{ and } \forall F \in \cE_K.
\]

Define the numerical gradient
 \beq\label{num-gradient}
\tilde \brho_{_\cT} = \nabla_h u_{_{\cT}} \quad \mbox{and} \quad
\tilde \brho_K = \nabla u_{_\cT}|_K, \quad \forall \, K  \in \cT . 
\eeq

For each edge $F\in\cE$, denote the $i$-th moment of a weighted average of the tangential components of the numerical 
gradient by 
\[
	S_{i,F}=\left\{
	\begin{array}{lll}
	\theta_\sF \displaystyle{\int_{F}} \left(  \tilde \brho_{{K_F^-}} \cdot \bt_F \right) L_{i,F} \,ds
	+\left(1-\theta_\sF\right) \displaystyle{\int_{F}} \left(  \tilde \brho_{{K_F^+}} \cdot \bt_F \right) L_{i,F} \,ds ,
	& \mbox{if} & F \in \cE_I,\\[6mm]
	0, & \mbox{if} & F \in \cE_D,\\[4mm]
	\displaystyle{\int_{F}} \left(  \tilde \brho_{{K_F^-}} \cdot \bt_F \right) L_{i,F} \,ds,
	& \mbox{if} & F \in \cE_N
	\end{array}
	\right.
\]
with the weight $\theta_\sF = \dfrac{\Lambda_F^-}{\Lambda_F^-+\Lambda_F^+}$ for $i=0,\, \cdots,\,k-1$. 
For each $K\in\cT$, define $\hat\brho_{K} \in N\!E^{k-1}(K)$ by
\beq\label{rho-construction}
\left \{
\begin{array}{lll}
		\displaystyle{\int_{F}} \big( \hat \brho_{K}  \cdot \bt_F \big )L_{i,F} \,ds= S_{i,F}, &\qquad
	\mbox{for } \,i=0, \cdots, k-1 \; \mbox{and} \;\forall\,\, F \in \cE_K,\\[6mm]
		\displaystyle{\int_{K}} \hat\brho_{K} \cdot \bzeta \,dx = \int_K \tilde \brho_{K} \cdot \bzeta\,dx ,
	& \qquad\forall \,\bzeta \in \mathbb{P}_{k-2}(K)^2.
	\end{array}
	\right.
\eeq
Then the recovered gradient $\hat\brho_{_\cT}$ is defined in $N\!E^{k-1}(\cT)$ such that 
\begin{equation}\label{gradient-recovery}
\hat\brho_{_\cT}\big|_K = \hat\brho_{K}, \quad\forall\,\,K\in\cT.
\end{equation}


\subsection{Equilibrated a posteriori error estimation for nonconforming solutions}
In section~\ref{odd-flux-recover}, we introduce an equilibrated flux recovery for the nonconforming elements of odd order. 
The construction is fully explicit. 
Let $\hat{\bsigma}_{_\cT} \in \S_f(\O)$ be the recovered flux defined in (\ref{global-flux}),
then the local indicator and the global estimator for the conforming error
are defined by
 \beq \label{estimators:cf_l}
	 \eta_{\sigma,K} = \|A^{-1/2} (\hat \bsigma_{_\cT} - \tilde \bsigma_{_\cT}) \|_{0,K}, \quad \forall \, K \in \cT 
	 \eeq
and
 \beq \label{estimators:cf_g}
	\eta_\sigma =\left( \sum_{K\in\cT} \eta^2_{\sigma,K}\right)^{1/2}
	=\|A^{-1/2} (\hat \bsigma_{_\cT} - \tilde \bsigma_{_\cT}) \|,  
\eeq 
respectively.

 In section~\ref{gradient recovery}, 
we recover the gradient 
in $H_D(\curll; \O)$ through averaging on each edge. This post-process procedure is also fully explicit. 
Let $\hat{\brho}_{_\cT} \in H_D(\curll; \O)$ be the recovered gradient defined in (\ref{gradient-recovery}), then the local indicator and the global estimator for the nonconforming error are defined by
 \beq \label{estimators:rho_l}
	 \eta_{\rho,K} = \|A^{1/2} (\hat \brho_{_\cT} - \tilde \brho_{_\cT}) \|_{0,K}, \quad \forall \, K \in \cT 
	 \eeq
and
 \beq \label{estimators:rho_g}
	\eta_\rho =\left( \sum_{K\in\cT} \eta^2_{\rho,K}\right)^{1/2}
	=\|A^{1/2} (\hat \brho_{_\cT} - \tilde \brho_{_\cT}) \|,  
\eeq 
respectively.

The local indicator and the global estimator for the nonconforming elements
are then defined by
\beq\label{estimators}
	\eta_K = \left( \eta_{\sigma,K}^2 + \eta_{\rho,K}^2\right)^{1/2}  
	\quad \mbox{and}\quad
	 \eta= \left( \sum_{K\in\cT} \eta_K^2 \right)^{1/2}
	 = \left( \eta_\sigma^2 + \eta_\rho^2 \right)^{1/2},
\eeq
respectively.

\subsection{Equilibrated a posteriori error estimation for DG solutions}
We first introduce the DG finite element method.
For any $K\in\cT$ and some $\alpha>0$, let 
 \[
 V^{1+\alpha}(K)=\{v\in H^{1+\a}(K)\,:\, \Delta\, v \in L^2(K)\}
 \]
and let 
 \[
  V^{1+\a}(\cT) :=\{v\, :\, v|_K\in V^{1+\alpha}(K), \quad \forall\, K\in\cT\}.
  \]
  We also denote the discontinuous finite element space $D_k$ of order $k $ (for $k \ge 0)$ by
  \[
  	D_k = \{v \in L^2(\O) : v|_K \in P^k(K), \quad \forall K \in \cT\}.
  \]
  For each $F \in \cE_I$, we define the following weights:
$\omega_F^\pm = \dfrac{\lambda_F^\mp}{\lambda_F^-+\lambda_F^+}$. In the weak formulation, we use the following weighted average:
\[
	\{v\}_w^F = \begin{cases} w_F^+ v_F^+ + w_F^- v_F^-, & F \in \cE_I,\\
	v ,& F \in \cE_D \cup \cE_N.
	\end{cases}
 \]
It is noted that the weighted average defined in the above way guarantees the robustness of the error estimation, see \cite{CaHeZh:17}.
 
Similar to \cite{CaHeZh:17} we introduce the following DG formulation for (\ref{pde}): find $u\in V^{1+\epsilon}(\cT)$ with $\epsilon >0$
 such that
 \beq\label{DGV}
 a_{dg}(u,\,v) = (f,\,v)   -
\left< g_N, v \right>_{\Gamma_N}, \quad\forall\,\, v \in V^{1+\epsilon}(\cT),
\eeq
where the bilinear form $a_{dg}(\cdot,\,\cdot)$ is given by
\begin{eqnarray*}
a_{dg}(u,v)&=&(A\nabla_h u,\nabla_h v)
 +\sum_{F\in\cE \setminus \cE_N}\int_F\gamma  \dfrac{\a_H}{h_F} \jump{u}
 \jump{v}\,ds \\[2mm] \nonumber
&&\quad -\sum_{F\in\cE  \setminus \cE_N}\int_F\{A\nabla
u\cdot\bn_F\}_{w}^F \jump{v}ds -
\sum_{F\in\cE \setminus \cE_N}\int_F \{A\nabla
v\cdot\bn_F\}_{w}^F\jump{u}ds.
\end{eqnarray*}
Here, 
$\a_H$ is the harmonic average of $\lambda$ over $F$, i.e., $\a_H = \dfrac{\lambda_F^+ \lambda_F^-}{\lambda_F^+  + \lambda_F^-}$
and $\gamma$ is a positive constant only depending on the shape of elements.
The discontinuous Galerkin finite element method  is then to
seek $u^{dg}_k \in D_k$ such that
\beq\label{problem_dg}
a_{dg}(u^{dg}_k,\, v) = (f,\,v)\quad \forall\, v\in D_k.
\eeq
For simplicity, we consider only this symmetric version of the interior penalty discontinuous Galerkin finite element method 
since its extension to other versions of discontinuous Galerkin approximations is straightforward.

Thanks to the complete discontinuity of the space $D_k$, an equilibrate flux for the DG solution $u_k^{dg}$ can be easily obtained. Here we present a formula similar to those introduced in \cite{Ai:07b,ern2007accurate,becker2016local}.
Recovering an equilibrate flux, $\hat \bsigma_k^{dg} \in RT^{k-1}(K)$, such that
 \begin{equation}\label{rt:flux:dg}
 \begin{split}
 	&( \hat\bsigma_k^{dg} , \btau)_K =-(A \nabla u_{k}^{dg} ,\btau )_K
	- \sum_{F \in \cE_K \cap \cE_I} \dfrac{1}{2} \mu_K\left< A \btau \cdot \bn_F, \jump{u_{k}^{dg}}\right>_F
	- \sum_{F \in \cE_K \cap \cE_D} \left< A \btau \cdot \bn_F, u_{k}^{dg}\right>_F
\end{split}
 \end{equation}
 for all $K \in \cT$ and for all $\btau \in \mathbb{P}_{k-2}(K)^d$,
 and that
 \beq 
\hat \bsigma_k^{dg} \cdot \bn_F = 
 \begin{cases}
 - \{A \nabla u_h \cdot \bn_F\} + \gamma h_F^{-1} \jump{u_h},& \forall F \in \cE_I,\\
  - A \nabla u_h \cdot \bn_F  + \gamma h_F^{-1} u_h, & \forall F \in \cE_D,\\
  g_N, & \forall F \in \cE_N.
\end{cases}
 \eeq
 It is easy to verify that the flux defined in (\ref{rt:flux:dg}) is equilibrate, i.e., $\nabla \cdot \hat \bsigma_k^{dg} = f_k$ where $f_k$ is the $L^2$ projection of $f$ onto the space of $D_k$.
 
 The recovery of the DG solution in the $H^1(\O)$ or the  $\mathring{H}_D(\curll; \O)$ spaces, again, suffers the lack of robustness. Similar to the nonconforming method, we also recover a gradient in the $H_D(\curll;\O)$ space. Let $\brho_k^{dg}$ be the recovered gradient for $u_k^{dg}$ based on the formulas in section \ref{gradient recovery}.
 The error indicators and estimators for $u_k^{dg}$ can then be similarly defined as in (\ref{estimators:rho_l})--(\ref{estimators}).


\section{Global reliability and local efficiency}
In this section, we establish the global reliability and efficiency for the error indicators and estimator defined in
in (\ref{estimators:rho_l})--(\ref{estimators}) for the NC elements of the odd orders. Similar robust results for DG solutions can be proved in the same way.

Let 
\[
\osc(f,K)= \dfrac{h_K}{ \sqrt{\lambda_K}} \| f -f_{k-1}\|_{0,K}
\quad\mbox{and}\quad
\osc(f,\cT) =\left( \sum_{K \in \cT} \osc(f,K)^2\right)^{1/2}.
\]

\begin{thm} {\em (Global Reliability)}
Let $u_{\cT}$ be the nonconforming solution to \upshape{(\ref{NC solution})}.
There exist constants $C_r$ and $C$ that is independent of the jump of the coefficient 
such that
\beq\label{rel}
	\|A^{1/2} \nabla_h (u - u_{_\cT})\|_{0,\O} \le \eta_\sigma + C_r\, \eta_\rho 
	+ C\,\osc(f,\cT)  .
\eeq
\end{thm}

\begin{proof}
The theorem is a direct result of Lemmas~\ref{lem-rel-cf} and \ref{lem-rel-nc}.
\end{proof}

Note that the global reliability bound in (\ref{rel}) does not
require the quasi-monotonicity assumption on the distribution of the diffusion coefficient $A(x)$.
The  reliability constant $C_r$ for the nonconforming error is independent of the jump of $A(x)$, but not equal to one.
This is due to the fact that the explicitly recovered gradient $\hat{\brho}_{_\cT}$ is not curl free.

In the following, we bound the conforming error above by the estimator $\eta_\sigma$
given in (\ref{estimators:cf_g}).

\begin{lem} \label{lem-rel-cf}
The global conforming error estimator, $\eta_\sigma$, given in {\em(\ref{estimators:cf_g})} is reliable, i.e.,
there exists a constant $C$ such that
\beq \label{global-reliability-cf}
	\inf_{ \btau \in \S_f(\O)} \|A^{1/2} ( \btau - \tilde \bsigma_{_\cT} )\| \le 
	 \eta_\sigma 
	 +C \,\osc(f,\cT) .
\eeq
\end{lem}

\begin{proof}
Let $\phi \in H_D^1(\O)$ be the conforming part of the Helmholtz decomposition of $u - u_{\cT}$.
By (\ref{HM:1}), integration by parts, and the assumption that $g|_F\in \mathbb{P}_{k-1}(F)$, we have 
\begin{equation} \label{rel:cf:a}
\begin{split}
& \inf_{\btau\in \S_f(\O)}\| A^{-1/2}\btau+ A^{1/2}\nabla_h u_\cT\|_{0,\O}^2\\
=&\|A^{1/2} \nabla \phi\|^2 
	=(A \nabla (u- u_{_\cT}), \nabla \phi) 
 =  (A \nabla u + \hat \bsigma_{_\cT}, \nabla \phi) 
  -( \hat \bsigma_{_\cT} - \tilde \bsigma_{_\cT}, \nabla \phi) \nonumber\\[2mm]
 &= (f- f_{k-1}, \phi) 
  - ( \hat \bsigma_{_\cT} - \tilde \bsigma_{_\cT}, \nabla \phi).
 \end{split}
   \end{equation}
Let $
	\bar \phi_K = \dfrac{1}{|K|} \int_K \phi \,dx.
$
It follows from the definitions of $f_{k-1}$ and the Cauchy-Schwarz and the Poincar\'{e} inequalities that
\begin{eqnarray*}
 && \sum_{K\in \cT} (f- f_{k-1}, \phi )_K 
=\sum_{K\in \cT} (f- f_{k-1}, \phi - \bar \phi_K)_K\\[2mm] 
    &\le&  C  \sum_{K\in \cT} \dfrac{h_K}{\lambda_K^{1/2} }
    \|f- f_{k-1}\|_{0,K} \| A^{1/2}\nabla \phi \|_{0,K} \\[2mm]
    &\le& C\,\osc(f,\cT) \|A^{1/2} \nabla \phi\|,
    \end{eqnarray*}
which, together with (\ref{rel:cf:a}) and the Cauchy-Schwartz inequality, leads to 
(\ref{global-reliability-cf}).
This completes the proof of the lemma.
\end{proof}

Since our recovered gradient is not in $\mathring{H}_D(\curll; \O)$, it is not straightforward to verify the reliability bound by Theorem \ref{Reliability}. However, it still plays a role in our reliability analysis.

\begin{lem} \label{lem-rel-nc}
The global nonconforming error estimator, $\eta_\rho$, given in {\em(\ref{estimators:rho_g})} is reliable, i.e.,
there exists a constant $C_r$ such that
\beq \label{rel:nc-b}
	\inf_{\btau \in \mathring{H}_D(\curll; \O)} \|A^{1/2} ( \btau - \nabla_h  u_{_\cT})\|
	\le C_r\, \eta_\rho.
\eeq
\end{lem}

\begin{proof}
By Lemma~\ref{lem-nc-error}, to show the validity of (\ref{rel:nc-b}), it then suffices to prove that
\beq \label{rel:nc-b-1}
	\lambda_F^{1/2} h_F^{-1/2} \|\jump{u_{_\cT}}\|_{0,F} \le 
	C \| A^{1/2} (\hat \brho_{_\cT} - \tilde \brho_{_\cT})\|_{0,\o_F} 
\eeq
for all $F \in \cE_I \cup \cE_D$. Note that $\jump{u_{_\cT}}|_F$ is an odd function for all $F \in \cE_I$.
Hence,  $\left\| \jump{ \tilde \brho_{_\cT} \cdot \bt_F} \right\|_{0,F}=0$ implies
$\left\|\jump{u_{_\cT}}\right\|_{0,F}=0$.
By the equivalence of norms in a finite dimensional space and the scaling argument, we have that
\beq\label{7.8}
	h_F^{-1/2}\|\jump{u_{_\cT}}\|_{0,F} \le
	C\, h_F^{1/2}\left\| \jump{ \tilde \brho_{_\cT} \cdot \bt_F} \right\|_{0,F}.
\eeq
Since $\hat\brho_{_\cT}\in H_D(\curll; \O)$, it then 
follows from the triangle, the trace, and the inverse inequalities that
\begin{eqnarray*}
	\left\| \jump{ \tilde \brho_{_\cT} \cdot \bt_F} \right\|_{0,F} 
	&=&  \left\| \jump{ (\tilde \brho_{_\cT} -\hat \brho_{_\cT}) \cdot \bt_F} \right\|_{0,F}  
	\le  \left\|  (\tilde \brho_{_\cT} -\hat \brho_{_\cT})|_{K_F^+} \cdot \bt_F \right\|_{0,F}
	   + \left\|  (\tilde \brho_{_\cT} -\hat \brho_{_\cT})|_{K_F^-} \cdot \bt_F \right\|_{0,F}\\[2mm]
	&\le& C\, h_F^{-1/2} \left( 
	\left\|  \tilde \brho_{_\cT} -\hat \brho_{_\cT}\right\|_{0,\o_F} 
	+h_F \| \curlt (\hat \brho_{_\cT} - \tilde \brho_{_\cT})\|_{0,\o_F}
	\right)\\[2mm]
	&\le& C\,h_F^{-1/2}  \left\|  \tilde \brho_{_\cT} -\hat \brho_{_\cT} \right\|_{0,\o_F} 
	\le C\, \lambda_F^{-1/2} h_F^{-1/2} \left\| A^{1/2}\left(  \tilde \brho_{_\cT} -\hat \brho_{_\cT} \right)\right\|_{0,\o_F} 
\end{eqnarray*}
for all $F \in \cE_I$,
which, together with (\ref{7.8}), implies (\ref{rel:nc-b-1}) and, hence,  (\ref{rel:nc-b}).
In the case that $F \in \cE_D$, (\ref{rel:nc-b-1}) can be proved in a similar fashion.
This completes the proof of the lemma.
\end{proof}

\subsection{Local Efficiency}

In this section, we establish local efficiency of the indicators $\eta_{\sigma,K}$ and $\eta_{\rho,K}$ 
defined in (\ref{estimators:cf_l}) and (\ref{estimators:rho_l}), respectively.


\begin{thm} {\em (Local Efficiency)}
For each $K \in \cT$, there exists a positive constant $C_e$ that is independent of the mesh size and the jump of the coefficient 
such that
\beq\label{eff}
	\eta_K \le C_e\, \left( \|A^{1/2} \nabla_h (u-u_{_\cT})\|_{0,\o_K} + \osc(f,K)\right),
\eeq
where $\o_K$ is the union of all elements that shares at least an edge with $K$.
\end{thm}


\begin{proof}
(\ref{eff}) is a direct consequence of Lemmas~\ref{thm:effi:cf}
and \ref{thm:effi:nc}.
\end{proof}

Note that the local efficiency bound in (\ref{eff}) holds regardless  the distribution of the diffusion coefficient $A(x)$.

\subsection{Local Efficiency for $\eta_{\sigma,K}$}

To establish local efficiency bound of $\eta_{\sigma,K}$, we introduce some auxiliary functions defined  
locally in $K$. To this end, for each edge $F\in \cE_K$, denote by $F'$ and $F''$ the other
two edges of $K$ such that $F, F'$, and $F''$ form counter-clockwise orientation. Without loss 
of generality, assume that $\mu_K \equiv 1$ on $\cE_K$. Let
\begin{equation}\label{w-F}
	w_\sF = \big(\hat \bsigma_K -\tilde \bsigma_{K} \big) \cdot \bn_K|_\sF \in \mathbb{P}_{k-1}(F),
	\quad a_\sF=w_\sF(\bs_\sF), \quad
	\mbox{and}  \quad
	b_\sF=w_\sF(\be_\sF).
\end{equation}
Define the auxiliary function corresponding to $F$, $\tilde w_\sF \in \mathbb{P}_{k}(K)$, such that
	\[
		\int_K \tilde w_\sF \, P_{j,K} \,dx=0, \quad \forall \, j=1, \cdots, m_k
	\]
	 and
	\[	
	\tilde w_\sF|_\sF = w_\sF + \gamma_\sF L_{k,F}, \quad
	\tilde w_\sF|_{_{F'}} = - \beta_{_{F}}L_{{k,F'}},
	\quad \mbox{and} \quad
	\tilde w_\sF|_{_{F''}} =\beta_{_{F}} L_{{k,F''}},
	\]
where $\gamma_\sF =\dfrac{a_\sF - b_\sF}{2}$ and $\beta_\sF =  \dfrac{a_\sF+b_\sF}{2}$.

\begin{lem}
For each $F \in \cE_K$, there exists a positive constant $C$ 
such that
\beq \label{effi:1}
	\|\tilde w_\sF\|_{0,K} \le C \,h_F^{1/2} \|w_\sF\|_{0,F}.
\eeq
\end{lem}

\begin{proof}	
By the Cauchy-Schwarz and the inverse inequalities, we have
\beq \label{effi:a1}
	\big| \gamma_\sF \big|= \Big| \dfrac{1}{2} \int_F w'_\sF \,ds \Big|
	 \le  \dfrac{h_F^{1/2}}{2}    \|w'_\sF\|_{0,F} 
	 \le C h_F^{-1/2} \|w_\sF\|_{0,F}.
\eeq
Approximation property and the inverse inequality give
\[
	\|w_\sF - \beta_\sF\|_{0,F} \le C h_F \|w'_\sF\|_{0,F} \le C \|w_\sF\|_{0,F},
\]
which, together with the triangle inequality, gives
\beq \label{effi:b1}
	|\beta_\sF| = h_F^{-1/2} \|\beta_\sF\|_{0,F} \le
	h_F^{-1/2} \big(  \|w_\sF - \beta_\sF\|_{0,F} + \| w_\sF\|_{0,F}  \big)
	\le C\, h_F^{-1/2} \| w_\sF\|_{0,F}  .
\eeq
Since $\|L_{k,F}\|_{0,F} \le h_F^{1/2}$ for all $F \in \cE_K$, by (\ref{effi:a1}) and (\ref{effi:b1}), we have that
\[
	\|\tilde w_\sF\|_{0,F} = \left(  \|w_\sF\|_{0,F}^2 +  \gamma_\sF^2\|L_{k,F}\|_{0,F}^2 \right)^{1/2}
	\le	C\, \|w_\sF\|_{0,F}
\]
and that
\[
	\|\tilde w_\sF\|_{0,F'} \le h_{_{F'}}^{1/2} |\beta_\sF| \le C\, \| w_\sF\|_{0,F} \quad
	\mbox{and} \quad
	\|\tilde w_\sF\|_{0,F''} \le h_{_{F''}}^{1/2} |\beta_\sF| \le C\, \| w_\sF\|_{0,F}.
\]
Now (\ref{effi:1}) is a direct consequence of the fact that
\[
	\|\tilde w_\sF\|_{0,K} \le C\, \sum_{F' \in \cE_K} h_{_{F'}}^{1/2} \|\tilde w_\sF\|_{0,F'}
\]
which follows from the equivalence of norms in a finite dimensional space, and the fact that
$\|\tilde w_\sF\|_{\partial K}=0$ implies $ \|\tilde w_\sF\|_K=0$. This completes
the proof of the lemma.
\end{proof}

\begin{lem}\label{thm:effi:cf}
There exists a positive constant $C$ 
such that
\beq \label{effi}
	\eta_{\sigma,K} \le C\, \left( \| A^{1/2} \nabla_h (u - u_{_\cT})\|_{0,K} +\osc(f,K)  \right) ,\quad \forall \, K \in \cT.
\eeq
\end{lem}

\begin{proof}
According to (\ref{rt:flux:b}), it is easy to see that
$\| \left(\hat \bsigma_K-\tilde \bsigma_{K} \right)\cdot \bn_F \|_{0,F}=0$ for all $F \in \cE_K$ implies that
$\| \hat \bsigma_K-\tilde \bsigma_K \|_{0,K} =0$. Hence, by the equivalence of norms in a finite dimensional
space, we have that
\beq\label{effi:1-a}
	\| \hat \bsigma_K-\tilde \bsigma_K \|_{0,K} \le C\,
	\sum_{F \in \cE_K} h_F^{1/2} \| \left(\hat \bsigma_K-\tilde \bsigma_K \right)\cdot \bn_F \|_{0,F}
	\le C \sum_{F \in \cE_K} h_F^{1/2} \| w_\sF \|_{0,F},
\eeq
where $w_\sF$ is defined in (\ref{w-F}).
By the orthogonality property of $\{L_{j,F}\}_{j=0}^k$ and 
	the definition of $\tilde w_\sF$, we have
	\[
		\left\|w_\sF \right\|_{0,F}^2 =
	\int_{\partial K} (\hat \bsigma_K -\tilde \bsigma_K) \cdot \bn \,\tilde w_\sF \,ds.
	\]
	It then follows from (\ref{equi:b}), integration by parts, the Cauchy-Schwarz inequality, and
	(\ref{effi:1}) that
		\begin{eqnarray*}
	\left\|w_\sF \right\|_{0,F}^2 
	&=& \int_K \tilde \bsigma_K \cdot \nabla \tilde w_\sF \,dx + \int_K f\, \tilde w_\sF \,dx
	-\int_K  \tilde \bsigma_K \cdot \nabla \tilde w_\sF \,dx 
	- \int_K \left( \nabla \cdot \tilde \bsigma_K \right)\, \tilde w_\sF dx\\[2mm]
	&=& \int_K \mathlarger(f - \nabla \cdot \tilde \bsigma_K\mathlarger) \,\tilde w_\sF \,dx
	\le  C\, h_F^{1/2}\large\|f - \nabla \cdot \tilde \bsigma_K \large\|_{0,K} \|w_\sF\|_{0,F},
	\end{eqnarray*}
	which implies
	\[
		\left\|w_\sF \right\|_{0,F} \le C h_F^{1/2}\large\|f - \nabla \cdot \tilde \bsigma_K \large\|_{0,K}.
	\]
	Together with (\ref{effi:1-a}), we have
	\[
		\eta_{\sigma,K}  \le  \lambda_K^{-1/2}\| \hat \bsigma_K-\tilde \bsigma_K \|_{0,K} \le
		 C\,\dfrac{h_K}{ \sqrt{\lambda_K}} \,
		\large\|f - \nabla \cdot \tilde \bsigma_K \large\|_{0,K}.
	\]
Now (\ref{effi}) is a direct consequence of the following efficiency bound of 
the element residual 
(see, e.g., \cite{BeVe:00}):
\[
	\dfrac{h_K}{\sqrt{\lambda_K}} \large\|f- \nabla \cdot \tilde \bsigma_K\|_K 
	\le C \left(	\big\|A^{1/2}  \nabla (u - u_{_\cT}) \big\|_{0,K}
	+ \dfrac{h_K}{\sqrt{\lambda_K}} \big\| f - f_{k-1}\big\|_{0,K } \right).
\]
This completes the proof of the theorem.

\end{proof}

\subsection{Local Efficiency for $\eta_{\rho,K}$}

In this section, we establish local efficiency bound for the nonconforming error indicator $\eta_{\rho,K}$ defined in
(\ref{estimators:rho_l}).

\begin{lem}\label{thm:effi:nc}
There exists a positive constant $C$ that is independent of the mesh size and the jump of the coefficient 
such that
\beq \label{effi:grad}
	\eta_{\rho,K} \le C \,\| A^{1/2} \nabla_h (u-u_{_\cT}) \|_{0,\o_K}, \quad \forall \, K \in \cT.
\eeq
\end{lem}

\begin{proof}
By (\ref{rho-construction}), it is easy to see that 
 $\| \left(\hat \brho_{K} - \tilde  \brho_{K} \right) \cdot \bt_F\|_{0,F}=0$ for all $F\in \cE_K$ implies
that $\|\hat \brho_{K} - \tilde  \brho{_K}\|_{0,K}=0$.
By the equivalence of norms in a finite dimensional space and the scaling argument,
we have
\beq \label{effi:a-4}
	\left\| \hat \brho_{K} - \tilde  \brho_{K}   \right\|_{0,K} \le C
	\sum_{F \in \cE_K} h_F^{1/2} \left\| \left(\hat \brho_{K} - \tilde  \brho_{K} \right)  \cdot \bt_F\right\|_{0,F}.
\eeq
Without loss of generality, assume that $K$ is an interior element. By (\ref{rho-construction}),
a direct calculation gives
\beq \label{effi:a-1}
	 \left(\hat \brho_{K} - \tilde \brho_K \right) \big|_F \cdot \bt_F =
	\left\{ 
	\begin{array}{lll}
	(\theta_F-1) \jump{ \tilde \brho \cdot \bt_F}|_\sF, & \mbox{if } & K=K_F^-,\\[4mm]
	\theta_\sF \jump{ \tilde \brho \cdot \bt_F}|_F,
	 & \mbox{if } & K=K_F^+
	 \end{array}
	 \right.
\eeq
for all $F \in \cE_K$.
It is also easy to verify that 
\beq  \label{effi:a-3}
	\left( \Lambda_F^- \right)^{1/2}  (1-\theta_F) \le 
	\left( \dfrac{ \Lambda_F^-   \Lambda_F^+ }{ \Lambda_F^- + \Lambda_F^+} \right)^{1/2} \quad
	\mbox{and} \quad
	\left( \Lambda_F^+ \right)^{1/2}  \theta_F \le 
	\left( \dfrac{ \Lambda_F^-   \Lambda_F^+ }{ \Lambda_F^- + \Lambda_F^+} \right)^{1/2}.
\eeq
Combining (\ref{effi:a-4}), (\ref{effi:a-1}), and (\ref{effi:a-3}) gives
\beq \label{effi:a-2}
\eta_{\rho,K}	\le \Lambda_K^{1/2} \big\| \hat \brho_{K} - \tilde \brho_K  \big\|_K \le
C \sum_{F \in \cE_K} 
	\left( \dfrac{ \Lambda_F^-   \Lambda_F^+ }{ \Lambda_F^- + \Lambda_F^+} \right)^{1/2}
	h_F^{1/2}\big \| \jump{ \tilde \brho_\cT \cdot \bt_F} \big\|_{0,F}.
\eeq
Now,  (\ref{effi:grad}) is a direct consequence of 
(\ref{effi:a-2}) and the following efficiency bound for the jump of tangential derivative on edges
\[
	\left( \dfrac{ \Lambda_F^-   \Lambda_F^+ }{ \Lambda_F^- + \Lambda_F^+} \right)^{1/2}
	h_F^{1/2}
	\big \| \jump{ \tilde \brho \cdot \bt_F} \big\|_{0,F} \le C  \,
	\big\|A^{1/2} \nabla (u-u_{_\cT})\big\|_{0,\o_F}
\]
for all $F \in \cE_I $.
This completes the proof of the lemma.
\end{proof}

\section{Numerical Result}
\setcounter{equation}{0}

In this section, we report numerical results on two test problems. The first one is on the 
Crouziex-Raviart nonconforming finite element approximation to 
the Kellogg benchmark problem \cite{Kel:74}. This is an interface problem in (\ref{pde}) with
$\O=(-1,1)^2$, $\Gamma_N=\emptyset$, $f=0$, 
 \[
 	A(x)=\left\{\begin{array}{ll}
 	161.4476387975881, & \quad\mbox{in }\, (0,1)^2\cup (-1,0)^2,\\[2mm]
 	1, & \quad\mbox{in }\,\O\setminus ([0,1]^2\cup [-1, 0]^2),
	 \end{array}\right.
 \]
 and the exact solution in the polar coordinates is given by 
$u(r,\theta)=r^{0.1}\mu(\theta)$, where $\mu(\theta)$ is a smooth function of $\theta$.

Starting with a coarse mesh, Figure \ref{mesh-kellogg} depicts the mesh when the relative error is less than $10\%$.
Here the relative error is defined as the ratio between the energy norm of the true error and the energy norm of the 
exact solution. Clearly, the mesh is centered around the singularity (the origin) and there is no over-refinement along interfaces. 
Figure \ref{error-kellogg} is the log-log plot of the energy norm of the true error and the global error estimator $\eta$ 
versus the total number of degrees of freedom. It can be observed that the error converges in an optimal order 
(very close to $-1/2$) and that the 
efficiency index, i.e., 
\[
\dfrac{\eta}{\|A^{1/2} \nabla_h (u-u_{_\cT})\|} 
\]
is close to one when the mesh is fine enough.
\begin{figure}[hb]
        \centering
    \begin{minipage}[htb]{0.48\linewidth}
        \includegraphics[trim=20 15 0 10,clip=true,width=1\textwidth,angle=0]{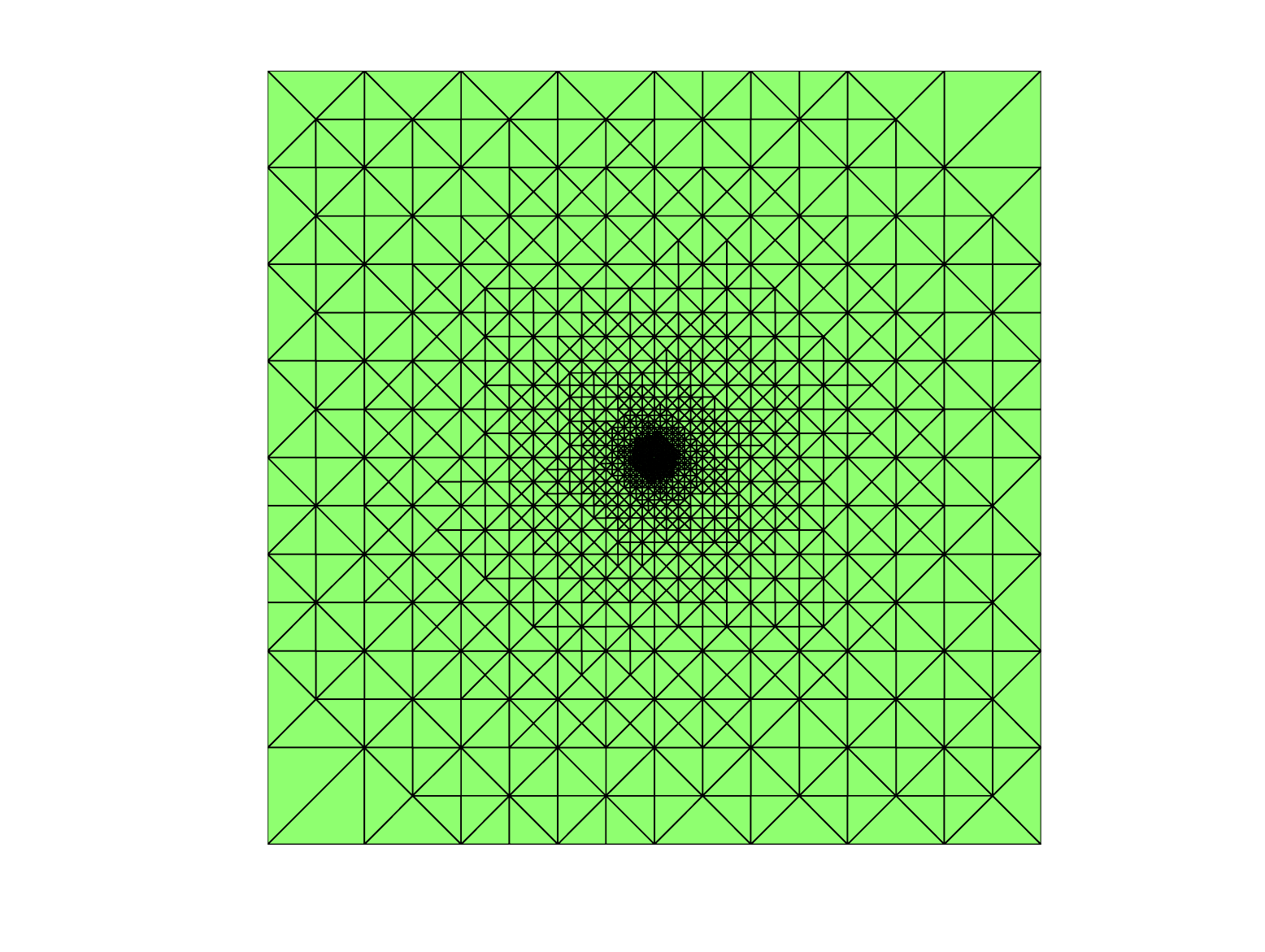}
        \vspace{-0.7cm}
        \caption{Kellogg problem: final mesh.}
    	\label{mesh-kellogg}
        \end{minipage}%
    \hspace{0.02\linewidth}
     \begin{minipage}[htb]{0.48\linewidth}
      \includegraphics[width=1\textwidth,angle=0]{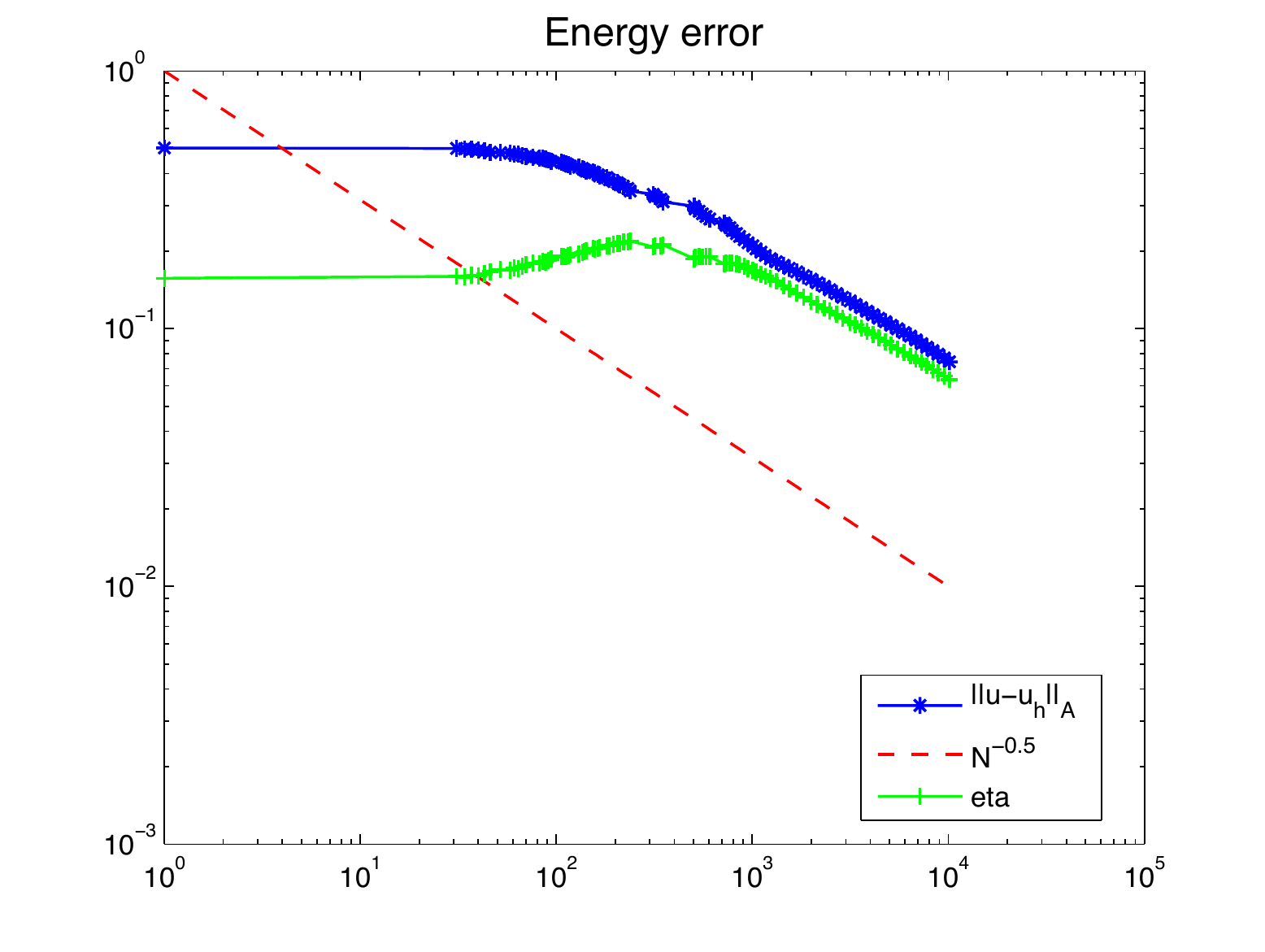}
        \vspace{-0.7cm}
        \caption{Error comparison.}
        \label{error-kellogg}
        \end{minipage}%
\end{figure}

With $f=0$ for the Kellogg problem, we note that $\eta_\sigma=0$, therefore, $\eta=\eta_{\rho}$.  Even though for the nonconforming error we recover a gradient that is not curl free, (thus we were not be able to prove that the reliability constant is $1$ for the nonconforming error) the numerics still shows the behavior of asymptotic exactness, i.e., when the mesh is fine enough the efficiency index is close to $1$.

For the second test problem, we consider a Poisson L-shaped problem that has a nonzero conforming error $\eta_\sigma$.
On the L-shaped domain $\O=[-1, \, 1]^2\setminus [0,\, 1]\times [-1,\,0]$, the Poisson problem ($A=I$) has 
the following exact solution
\[
	u(r,\theta) = r^{2/3} \sin((2 \theta+\pi)/3) + r^2/2.
\]
The numerics is based on the Crouziex-Raviart finite element approximation.
With the relative error being less than $0.75\%$, the final mesh generated the adaptive 
mesh refinement algorithm is depicted in Figure \ref{mesh-lshape}. Clearly, the mesh is relatively centered around the singularity (origin). 
Comparison of the true error and the estimator is presented in Figure \ref{error-lshape}. 
It is obvious that the error converges in an optimal order (very close to $-1/2$) and that the 
efficiency index is very close to $1$ for all iterations.

\begin{figure}[hb]
        \centering
    \begin{minipage}[htb]{0.48\linewidth}
        \includegraphics[trim=20 15 0 10,clip=true,width=1\textwidth,angle=0]{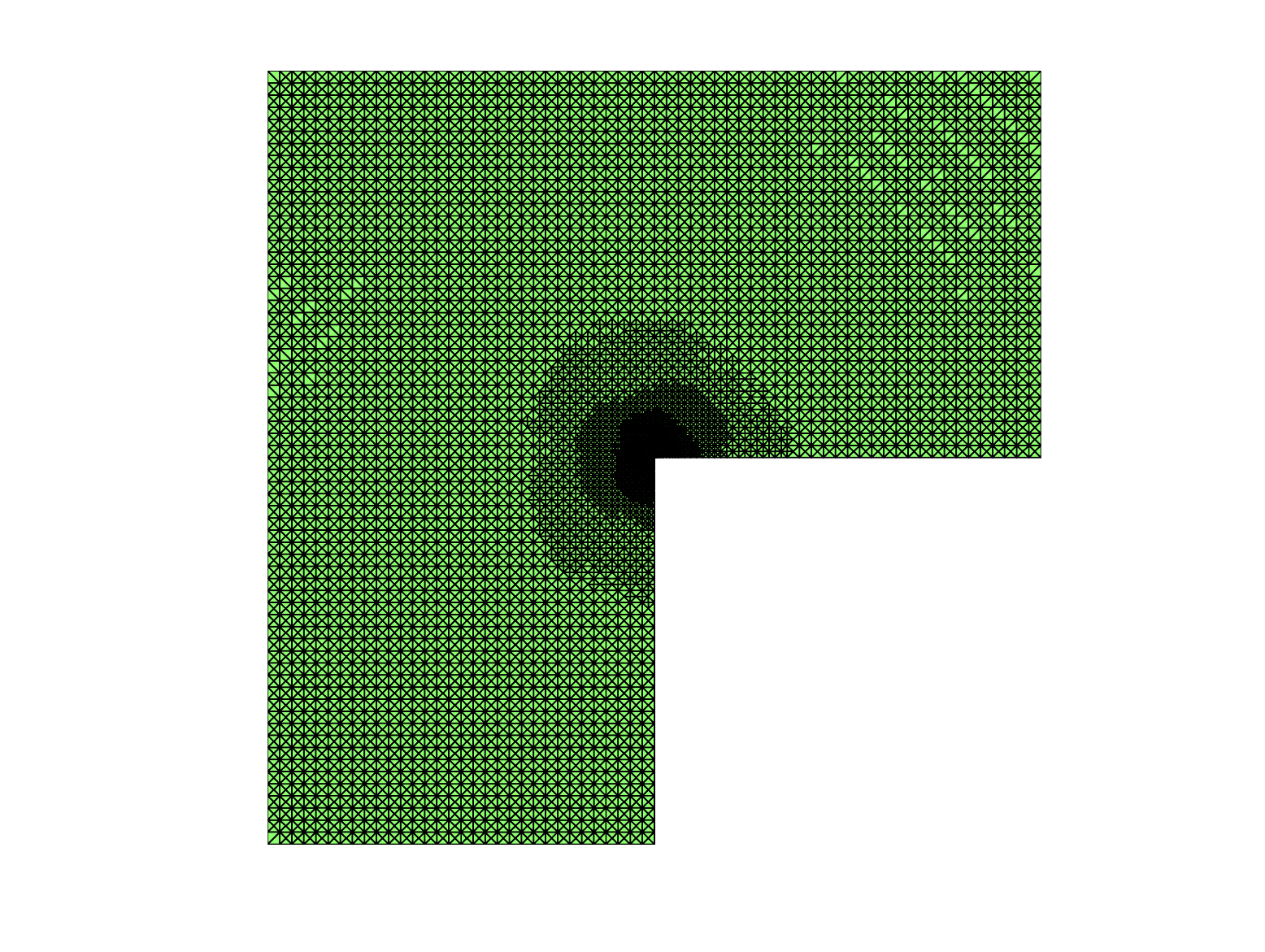}
        \vspace{-0.7cm}
        \caption{L-shape problem: final mesh.}
    	\label{mesh-lshape}
        \end{minipage}%
    \hspace{0.02\linewidth}
     \begin{minipage}[htb]{0.48\linewidth}
     \includegraphics[width=1\textwidth,angle=0]{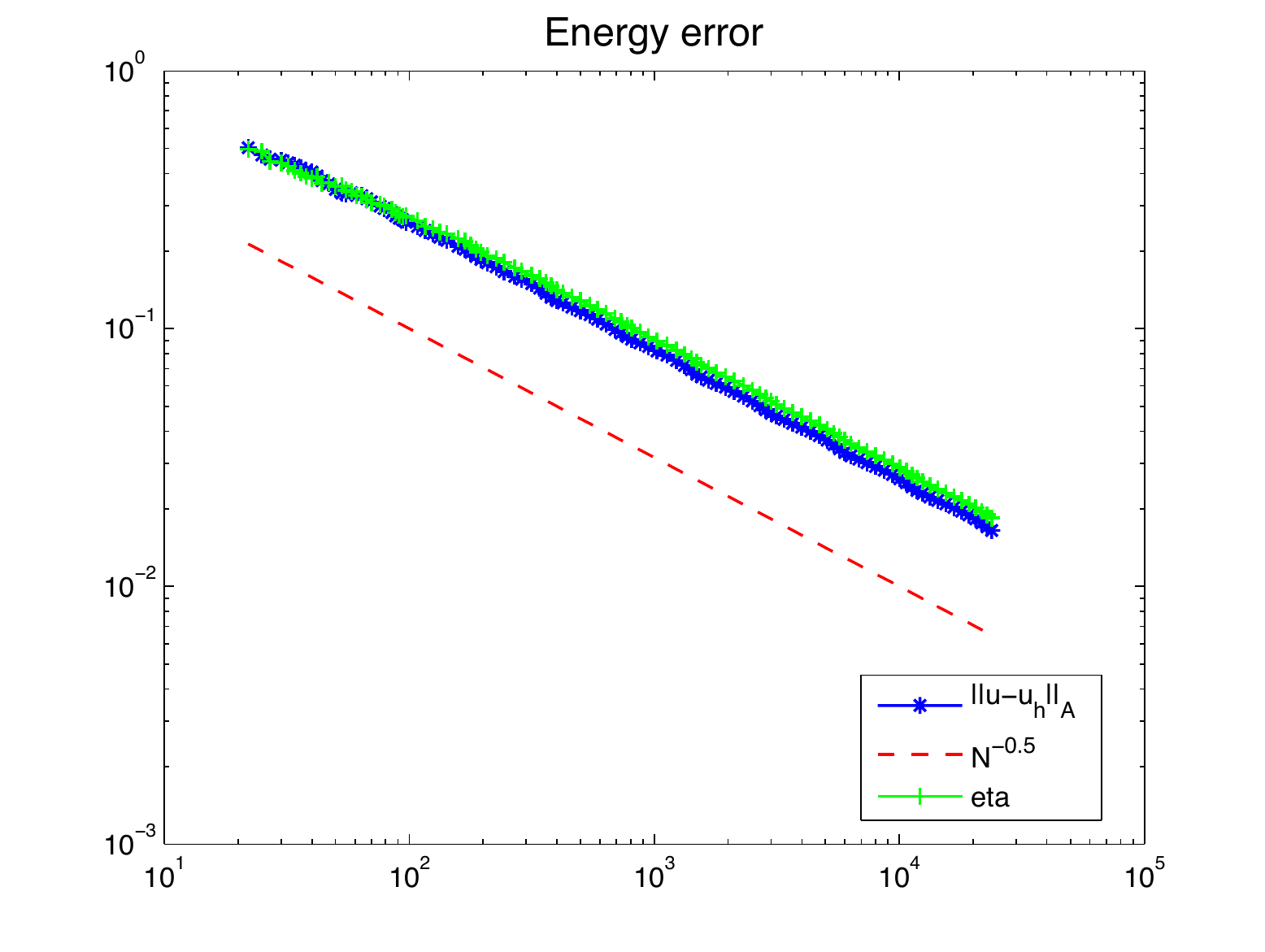}
        \vspace{-0.7cm}
        \caption{Error comparison.}
        \label{error-lshape}
        \end{minipage}%
\end{figure}

\end{document}

%% file: defs.tex
\newcommand{\BX}{{\bf X}}
\newcommand{\cv}{{\cal V}}
\newcommand{\cW}{{\cal W}}
\newcommand{\co}{{\cal O}}

\renewcommand{\theequation}{\thesection.\arabic{equation}}
\def\@eqnnum{{\reset@font\rm (\theequation)}}

\newtheorem{definition}{Definition}[section]

\def\abstract{
\advance \rightskip by 10mm
\advance \leftskip by 10mm
\vspace{-0.8em}
\noindent
\small{\bf Abstract.}
}
\def\endabstract{\par\normalsize\rm}

\def\Xint#1{\mathchoice
{\XXint\displaystyle\textstyle{#1}}%
{\XXint\textstyle\scriptstyle{#1}}%
{\XXint\scriptstyle\scriptscriptstyle{#1}}%
{\XXint\scriptscriptstyle\scriptscriptstyle{#1}}%
\!\int}
\def\XXint#1#2#3{{\setbox0=\hbox{$#1{#2#3}{\int}$}
\vcenter{\hbox{$#2#3$}}\kern-.5\wd0}}
\def\ddashint{\Xint=}
\def\dashint{\Xint-}

\def\a{\alpha}
\def\b{\beta}
\def\d{\delta}\def\D{\Delta}
\def\e{\epsilon}
\def\g{\gamma}\def\G{\Gamma}
\def\k{\kappa}
\def\lam{\lambda}\def\Lam{\Lambda}
\renewcommand\o{\omega}\renewcommand\O{\Omega}
\def\s{\sigma}\def\S{\Sigma}
\renewcommand\t{\theta}\def\vt{\vartheta}
\newcommand{\vphi}{\varphi}
\def\z{\zeta}

\newcommand{\tsigma}{\tilde{\s}}
\newcommand{\tbsigma}{\tilde{\bsigma}}
\def\te{\tilde{\e}}
\def\tu{\tilde{u}}

\newcommand{\bchi}{\mbox{\boldmath$\chi$}}
\newcommand{\bdelta}{\mbox{\boldmath$\delta$}}
\newcommand{\bepsilon}{\mbox{\boldmath$\epsilon$}}
\newcommand{\bfeta}{\mbox{\boldmath$\eta$}}
\newcommand{\bgamma}{\mbox{\boldmath$\gamma$}}
\newcommand{\bomega}{\mbox{\boldmath$\omega$}}
\newcommand{\bvphi}{\mbox{\boldmath$\varphi$}}
\newcommand{\bphi}{\mbox{\boldmath$\phi$}}
\newcommand{\bPhi}{\mbox{\boldmath$\Phi$}}
\newcommand{\bpsi}{\mbox{\boldmath$\psi$}}
\newcommand{\bPsi}{\mbox{\boldmath$\Psi$}}
\newcommand{\bsigma}{\mbox{\boldmath$\sigma$}}
\newcommand{\btau}{\mbox{\boldmath$\tau$}}
\newcommand{\bxi}{\mbox{\boldmath$\xi$}}
\newcommand{\brho}{\mbox{\boldmath$\rho$}}
\newcommand{\bbeta}{\mbox{\boldmath$\beta$}}
\newcommand{\bzeta}{\mbox{\boldmath$\zeta$}}

\def\bk{\boldsymbol{\kappa}}
\def\bmu{\boldsymbol\mu}
\def\bxi{\boldsymbol{\xi}}
\def\bz{\boldsymbol{\zeta}}

\def\ba{{\bf a}}
\def\bb{{\bf b}}
\def\bc{{\bf c}}
\def\be{{\bf e}}
\def\bff{{\bf f}}
\def\bg{{\bf g}}
\def\bn{{\bf n}}
\def\bp{{\bf p}}
\def\bq{{\bf q}}
\def\bs{{\bf s}}
\def\bt{{\bf t}}
\def\bu{{\bf u}}
\def\bv{{\bf v}}
\def\bw{{\bf w}}
\def\bx{{\bf x}}
\def\by{{\bf y}}
\def\bzz{{\bf z}}

\def\bD{{\bf D}}
\def\bE{{\bf E}}
\def\bF{{\bf F}}
\def\bH{{\bf H}}
\def\bJ{{\bf J}}
\def\bV{{\bf V}}
\def\bU{{\bf U}}
\def\bW{{\bf W}}
\def\bX{{\bf X}}
\def\bY{{\bf Y}}

\def\cA{{\cal A}}
\def\cC{{\cal C}}
\def\cD{{\cal D}}
\def\cE{{\cal E}}
\def\cF{{\cal F}}
\def\cG{{\cal G}}
\def\cI{{\cal I}}
\def\cJ{{\cal J}}
\def\cK{{\cal K}}
\def\cL{{\cal L}}
\def\cO{{\cal O}}
\def\cP{{\cal P}}
\def\cQ{{\cal Q}}
\def\cR{{\cal R}}
\def\cS{{\cal \Sigma}}
\def\cT{{\cal T}}
\def\cU{{\cal U}}
\def\cV{{\cal V}}

\def\scT{{_\cT}}
\def\sD{{_D}}
\def\sE{{_E}}
\def\sF{{_F}}
\def\sFz{{_{F_z}}}
\def\sK{{_K}}
\def\sI{{_I}}
\def\sb{{_b}}
\def\sN{{_N}}

\def\curl{{\mbox{curl}\ }}
\def\rot{{\mbox{rot}\ }}
\def\BPI{{\bf \Pi}}

\def\bigintF{{\displaystyle \int_F}}
\def\bigintK{{\displaystyle \int_K}}

\def\cth{\cT_h}
\def\ctH{\cT_H}

\def\tJ{\tilde{\J}}

\def\hK{\widehat{K}}
\def\hx{\widehat{x}}
\def\hy{\widehat{y}}
\def\bhv{\widehat{\bv}}

\def\l{\ell}
\def\bl{\boldsymbol{\ell}}
\def\col{\colon}
\def\f12{\frac12}
\def\dfrac{\displaystyle\frac}
\def\dint{\displaystyle\int}
\def\nab{\nabla}
\def\p{\partial}
\def\sm{\setminus}
\def\dsum{\displaystyle\sum}
\newcommand{\pp}[2]{\frac{\partial {#1}}{\partial {#2}}}
\def\bzero{{\bf 0}}

\def\divv{\nab\cdot}
\def\divx{\nab_x\cdot}
\def\divtx{\nab_{t,x}\cdot}
\def\nabx{\nab_x}

\newcommand{\curlt}{{\nabla \times}}
\newcommand{\gperp}{\nabla^{\perp}}
\newcommand{\gradt}{\nabla\cdot}

\def\forallqq{\quad\forall\,}
\def\aph{A^{1/2}}
\def\amh{A^{-1/2}}

\def\osc{{\rm osc \, }}

\def\Im{{\rm Im}}
\newcommand{\tr}{{\rm tr}}
\def\divvr{{\rm div}}
\def\curllr{{\rm curl}}
\def\curll{{\rm curl}}
\newcommand{\bgrad}{{\bf grad}}
\newcommand\diam{\mathrm{diam\,}}
\renewcommand\Im{\mathrm{Im\,}}
\def\Span{\mbox{Span}}
\def\supp{\mbox{supp\,}}
\newcommand{\trace}{{\rm trace}}

\newcommand{\tri}{|\!|\!|}
\newcommand{\ljump}{\lbrack\!\lbrack}
\newcommand{\rjump}{\rbrack\!\rbrack}
\newcommand{\bdm}{\begin{displaymath}}
\newcommand{\edm}{\end{displaymath}}
\newcommand{\beq}{\begin{equation}}
\newcommand{\eeq}{\end{equation}}
\newcommand{\beqa}{\begin{eqnarray}}
\newcommand{\eeqa}{\end{eqnarray}}
\newcommand{\beqas}{\begin{eqnarray*}}
\newcommand{\eeqas}{\end{eqnarray*}}
\newcommand{\ul}{\underline}
\newcommand{\wh}{\widehat}
\newcommand{\la}{\langle}
\newcommand{\ra}{\rangle}

\newcommand{\Lt}{L^2(\Omega)}
\newcommand{\Lts}{L^2(\Omega)^2}
\newcommand{\Ltc}{L^2(\Omega)^3}
\newcommand{\Ho}{H^1(\Omega)}
\newcommand{\Hoh}{H^1(\wh{\Omega})}
\newcommand{\Hoi}{H^1(\Omega_i)}
\newcommand{\Hos}{H^1(\Omega)^2}
\newcommand{\Hoc}{H^1(\Omega)^3}
\newcommand{\Hoch}{H^1(\wh{\Omega})^3}
\newcommand{\Hoci}{H^1(\Omega_i)^3}
\newcommand{\Hoz}{H^1_0(\Omega)}
\newcommand{\Ht}{H^2(\Omega)}
\newcommand{\Hti}{H^2(\Omega_i)}
\newcommand{\Hts}{H^2(\Omega)^2}
\newcommand{\Htc}{H^2(\Omega)^3}
\newcommand{\Htz}{H^0(\Omega)}
\newcommand{\Hh}{H^{1/2}(\Gamma)}
\newcommand{\Hhi}{H^{1/2}(\Gamma_i)}
\newcommand{\Hmh}{H^{-1/2}(\Gamma)}
\newcommand{\Hdiv}{H(\divvr;\,\Omega)}
\newcommand{\Hdivh}{H(\divv;\,\wh \Omega)}
\newcommand{\hcurl}{H(\curl\,A;\,\Omega)}
\newcommand{\Hcurl}{H(\curll\,A;\,\Omega)}
\newcommand{\Hcrl}{H(\curll\,;\,\Omega)}
\newcommand{\hcrl}{H(\curl\,;\,\Omega)}
\newcommand{\Hcrlh}{H(\curll\,;\,\wh\Omega)}
\newcommand{\hcrlh}{H(\curl\,;\,\wh\Omega)}
\newcommand{\Wdiv}{\BW_0(\mbox{\divv}\,;\,\Omega)}
\newcommand{\Wcurl}{\BW_0(\mbox{\curl}\,A;\,\Omega)}
\newcommand{\WcrossV}{\BW \times V}